\documentclass[11pt,oneside]{article}

\usepackage{latexsym}
\usepackage{shortvrb}
\usepackage{graphicx}
\usepackage{fancyhdr}
\usepackage{amsfonts}
\usepackage{amsmath}
\usepackage{amssymb}
\usepackage{mathrsfs}
\usepackage{makeidx}
\usepackage[all]{xy}
\usepackage{amsthm}
\usepackage{tikz}
\usepackage{cite}
\usepackage{hyperref}

\hypersetup{
    colorlinks=true,
    linkcolor=blue,
    filecolor=magenta,      
    urlcolor=cyan,
}

\usetikzlibrary{matrix}

\usepackage[T1]{fontenc}
\usepackage{amsfonts}
\usepackage{amsmath}

\newtheorem{thm}{Theorem}[section]
\newtheorem{lemma}[thm]{Lemma}
\newtheorem{prop}[thm]{Proposition}
\newtheorem{ex}[thm]{Example}
\newtheorem{corollary}[thm]{Corollary}
\newtheorem{notation}[thm]{Notation}

\theoremstyle{definition}
\newtheorem{definition}[thm]{Definition}
\newtheorem{obs}[thm]{Observation}

\newtheorem{prob}[thm]{Problem}

\newcommand{\blackged}{\hfill$\blacksquare$}
\newcommand{\whiteged}{\hfill$\square$}
\newcounter{proofcount}

\renewenvironment{proof}[1][\proofname.]{\par
  \ifnum \theproofcount>0 \pushQED{\whiteged} \else \pushQED{\blackged} \fi%
  \refstepcounter{proofcount}
  \normalfont 
  \trivlist
  \item[\hskip\labelsep
        \itshape
    {\bf\em #1}]\ignorespaces
}{%
  \addtocounter{proofcount}{-1}
  \popQED\endtrivlist
}

\begin{document}

\begin{center}
\textbf{Free Globularily Generated Double Categories I}
\end{center}

\begin{center}
\textit{Juan Orendain}
\end{center}

\noindent \textit{Abstract:} This is the first part of a two paper series studying free globularily generated double categories. In this first installment we introduce the free globularily generated double category construction. The free globularily generated double category construction canonically associates to every bicategory together with a possible category of vertical morphisms, a double category fixing this set of initial data in a free and minimal way. We use the free globularily generated double category to study length, free products, and problems of internalization. We use the free globularily generated double category construction to provide formal functorial extensions of the Haagerup standard form construction and the Connes fusion operation to inclusions of factors of not-necessarily finite Jones index.

\tableofcontents

\section{Introduction}\label{s1}

\noindent  Double categories were introduced by Ehresmann in \cite{Ehr1}. Bicategories were later introduced by B\'enabou in \cite{BenabouBicats}. Both double categories and bicategories express the notion of a higher categorical structure of second order, each with its advantages and disadvantages. Double categories and bicategories relate in different ways.

Every double category admits an underlying bicategory, its horizontal bicategory. The horizontal bicategory $HC$ of a double category $C$ 'flattens' $C$ by discarding vertical morphisms and only considering globular squares. There are several structures transferring vertical information on a double category to its horizontal bicategory, e.g. connection pairs \cite{BrownMosa}, thin structures \cite{BrownSpencer}, and foldings and cofoldings \cite{BrownSpencer2} among others. A great deal of information about a double category can be reduced to information about its horizontal bicategory under the assumption of the existence of such structures, see \cite{FioreGambinoKock} for example. 

Bicategories on the other hand 'lift' to double categories through several different constructions, examples of which are the Ehresmann double category of quintets construction \cite{Ehr3}, the double category of adjoints construction \cite{Palmquist}, the double category of spans construction \cite{PronkDawsonPare}, the construction of framed bicategories through monoidal fibrations of \cite{SchulmanFramed} and the construction of the double category of semisimple von Neumann algebras and finite morphisms of \cite{Bartels1,Bartels2}. These construction follows different methods and the resulting double categories express different aspects of the bicategories they lift. In all cases one starts with a bicategory as initial set of data together with a choice of vertical morphisms, which serve as a 'direction' towards which one lifts. In all cases one ends up with a double category having relevant information about the initial bicategory and the collection of vertical morphisms, and relating to this initial set of data through horizontalization.

We are interested general constructions of the type described above. This is the first part of a two paper series studying the free globularily generated double category construction. The free globularily generated double category construction canonically associates to every bicategory, together with a direction towards which to lift, i.e. together with a category of vertical morphisms, a double category. This double category fixes the initial set of data and is minimal with respect to this property. In this paper we provide a detailed construction of the free globularily generated double category associated to a decorated bicategory and we apply this construction to problems of existence of internalizations and to the concept of length of a double category. We now present a more detailed account of the contents of this paper.

\

\noindent \textit{The problem of existence of internalizations} 

\

\noindent We are interested in the following situation. Given a bicategory $\mathcal{B}$ we will say that a category $\mathcal{B}^*$ is a decoration of $\mathcal{B}$ if the collection of 0-cells of $\mathcal{B}$ is equal to the collection of objects of $\mathcal{B}^*$. In this case we say that the pair $(\mathcal{B}^*,\mathcal{B})$ is a decorated bicategory. We think of decorated bicategories as bicategories together with an orthogonal direction, provided by the corresponding decoration, towards which to lift $\mathcal{B}$ to a double category. Given a double category $C$ the pair $(C_0,HC)$ where $C_0$ is the category of objects of $C$, is a decorated bicategory. We write $H^*C$ for this decorated bicategory. We call $H^*C$ the decorated horizontalization of $C$. We consider the following problem.

\begin{prob}\label{prob}
Let $(\mathcal{B}^*,\mathcal{B})$ be a decorated bicategory. Find double categories $C$ such that $H^*C=\mathcal{B}$.
\end{prob}

\noindent Given a decorated bicategory $(\mathcal{B}^*,\mathcal{B})$ we say that a solution $C$ to problem \ref{prob} for $(\mathcal{B}^*,\mathcal{B})$ is an internalization of $(\mathcal{B}^*,\mathcal{B})$. We are thus interested in finding internalizations to decorated bicategories. Pictorially, we are intereted in the following situation: Given a set of 2-dimensional cells of the form

\begin{center}

\begin{tikzpicture}
\matrix(m)[matrix of math nodes, row sep=4em, column sep=4em,text height=1.5ex, text depth=0.25ex]
{\bullet&\bullet\\};
\path[->,font=\scriptsize]
(m-1-1) edge [bend left=45] node [above]{$\alpha$}(m-1-2)
        edge [bend right=45] node [below]{$\beta$}(m-1-2)
        edge [white]node[black][fill=white]{$\varphi$}(m-1-2) ;
\end{tikzpicture}
\end{center}

\noindent forming a bicategory, and given a collection of vertical arrows of the form:

\begin{center}

\begin{tikzpicture}
\matrix(m)[matrix of math nodes, row sep=4em, column sep=4em,text height=1.5ex, text depth=0.25ex]
{\bullet\\
\bullet\\};
\path[->,font=\scriptsize]
(m-1-1) edge node [left]{$f,g \ \mbox{etc.}$}(m-2-1);
\end{tikzpicture}
\end{center}

\noindent forming a category, such that the endpoints of these arrows are the same as the verteces of the above globular diagrams, we consider boundaries of hollow squares of the form:

\begin{center}

\begin{tikzpicture}
\matrix(m)[matrix of math nodes, row sep=4em, column sep=4em,text height=1.5ex, text depth=0.25ex]
{\bullet&\bullet\\
\bullet&\bullet\\};
\path[->,font=\scriptsize]
(m-1-1) edge node [above]{$\alpha$} (m-1-2)
        edge node [left]{$f$} (m-2-1)
(m-2-1) edge node[below]{$\beta$} (m-2-2)			
(m-1-2) edge node[right]{$g$} (m-2-2);
\end{tikzpicture}

\end{center}

\noindent formed by horizontal edges of globular diagrams and the set of decoration arrows provided as our set of initial conditions. Identifying globular diagrams as above with squares of the form:

\begin{center}

\begin{tikzpicture}
\matrix(m)[matrix of math nodes, row sep=4em, column sep=4em,text height=1.5ex, text depth=0.25ex]
{\bullet&\bullet\\
\bullet&\bullet\\};
\path[->,font=\scriptsize]
(m-1-1) edge node [above]{$\alpha$} (m-1-2)
        edge[blue] node [black][left]{$id$} (m-2-1)
(m-2-1) edge node[below]{$\beta$} (m-2-2)			
(m-1-2) edge [blue]node[right][black]{$id$} (m-2-2)
(m-1-1) edge [white]node[black][fill=white]{$\varphi$}(m-2-2);
\end{tikzpicture}

\end{center}

\noindent and formally associating to every vertical arrow as above a unique identity asquare as:

\begin{center}

\begin{tikzpicture}
\matrix(m)[matrix of math nodes, row sep=4em, column sep=4em,text height=1.5ex, text depth=0.25ex]
{\bullet&\bullet\\
\bullet&\bullet\\};
\path[->,font=\scriptsize]
(m-1-1) edge [red]node [black][above]{$id$} (m-1-2)
        edge node [left]{$f$} (m-2-1)
(m-2-1) edge [red]node[black][below]{$id$} (m-2-2)			
(m-1-2) edge node[right]{$f$} (m-2-2)
(m-1-1) edge [white]node[black][fill=white]{$i_f$}(m-2-2);
\end{tikzpicture}

\end{center}

\noindent problem \ref{prob} asks about coherent ways to fill boundaries of hollow squares as above with globular and horizontal squares associated to our set of initial data in such a way that vertical arrows and globular squares are fixed and such that the resulting structure forms a double category. We regard such problems as formal versions of arguments of 'filling squares' classically considered in nonabelian algebraic topology, see \cite{BrownBook}. One of the author's motivations for studying such problems is the problem of existence of a compatible pair of tensor functors associating to every von Neumann algebra $A$ its Haagerup standard form and associating to every horizontally compatible pair of Hilbert bimodules $_AH_B,_BK_C$ the corresponding fusion Hilbert bimodule $_AH\boxtimes_BK_C$ respectively. Such compatible pair of functors should provide the pair formed by the category of von Neumann algebras and their morphisms and the category of Hilbert bimodules and equivariant intertwining operators with the structure of a category internal to tensor categories. The methods developed in the present series of papers provide partial solutions to this problem.

\

\noindent \textit{A case for globularily generated double categories}

\

\noindent The situation described above motivated the author to introduce the concept of globularily generated double category in \cite{yo1}. We say that a double category is globularily generated if it is generated by its collection of globular squares. Pictorially, a double category $C$ is globularily generated if every square in $C$ can be written as horizontal and vertical compositions of squares of the form:

\begin{center}

\begin{tikzpicture}
\matrix(m)[matrix of math nodes, row sep=4em, column sep=4em,text height=1.5ex, text depth=0.25ex]
{\bullet&\bullet&\bullet&\bullet\\
\bullet&\bullet&\bullet&\bullet\\};
\path[->,font=\scriptsize]
(m-1-1) edge node [above]{$\alpha$} (m-1-2)
        edge[blue] node [black][left]{$id$} (m-2-1)
(m-2-1) edge node [below]{$\beta$} (m-2-2)			
(m-1-2) edge [blue] node[black][right] {$id$} (m-2-2)
(m-1-1) edge [white] node[black][fill=white]{$\varphi$} (m-2-2)

(m-1-3) edge[red] node [black][above]{$id$} (m-1-4)
        edge node [left]{$f$} (m-2-3)
(m-2-3) edge [red]node[black][below]{$id$} (m-2-4)			
(m-1-4) edge node [right]{$f$} (m-2-4)
(m-1-3) edge [white] node[black][fill=white]{$i_f$} (m-2-4);
\end{tikzpicture}

\end{center}

\noindent Given a double category $C$ we write $\gamma C$ for the sub-double category of $C$ generated by squares as above. $\gamma C$ is globularily generated, it satisfies the equation $H^*C=H^*\gamma C$ and it is the minimal sub-double category of $C$ satisfying this equation. We call $\gamma C$ the globularily generated piece of $C$. A double category $C$ is globularily generated if and only if there are no proper sub-double categories $D$ of $C$ such that $H^*C=H^*D$. Globularily generated double categories are thus precisely the minimal solutions to problem \ref{prob}. This can be expressed categorically as follows: Write \textbf{dCat}, \textbf{gCat} and \textbf{bCat}$^*$ for the category of double categories and double functors, for the subcategory of \textbf{dCat} generated by globularily generated double categories, and for the category of decorated bicategories and decorated pseudofunctors respectively. The globularily generated piece construction extends to a reflector (2-reflector in fact) $\gamma$ of \textbf{gCat} in \textbf{dCat}. It is not difficult to see that this implies that $\gamma$ is in fact a Grothendieck fibration. Moreover, the decorated horizontal bicategory construction extends to a functor $H^*$ from \textbf{dCat} to \textbf{bCat}$^*$, which by the comments above is easily seen to be constant on the fibers of $\gamma$. We obtain a commutative triangle:

\begin{center}

\begin{tikzpicture}
\matrix(m)[matrix of math nodes, row sep=4em, column sep=3em,text height=1.5ex, text depth=0.25ex]
{\mbox{\textbf{dCat}}&&\mbox{\textbf{bCat}}^*\\
&\mbox{\textbf{gCat}}&\\};
\path[->,font=\scriptsize]
(m-1-1) edge node [above]{$H^*$} (m-1-3)
        edge node [left]{$\gamma$} (m-2-2)
(m-2-2) edge node[right]{$H^*\restriction_{\mbox{\textbf{gCat}}}$} (m-1-3)			;
\end{tikzpicture}

\end{center}

\noindent We thus think of double categories as being parametrized, or bundled, by globularily generated double categories. The relevant information about problem \ref{prob} is contained in the bases of this fibration. We summarize this by saying that finding solutions to problem \ref{prob} is equivalent to finding globularily generated solutions. We believe this justifies the study of globularily generated double categories.

Globularily generated double categories admit intrinsical structure that makes them, to some extent, easy to describe. The category of squares $C_1$ of a globularily generated double category $C$ admits an expression as a limit $\varinjlim V^k_C$ of a chain of categories $V^1_C\subseteq V^2_C\subseteq \dots$ defined inductively by setting $V^1_C$ as the subcategory of $C_1$ generated by squares as above, and by setting $V^k_C$ as the subcategory of $C$ generated by horizontal compositions of morphisms in $V^{k-1}_C$ for every $k>1$. We call this chain of categories the vertical filtration of $C$. The vertical filtration allows us to define numerical invariants for double categories. We say that a square $\varphi$ in a globularily generated double category $C$ is of length $k$, $\ell\varphi$ in symbols, if $\varphi$ is a morphism in $V^k_C$ but not a morphism in $V^{k-1}_C$. We define the length of a double category $C$, which we write $\ell C$, as the supremum of lengths of squares in $\gamma C$. The only examples of globularily generated double categories studied so far, i.e. trivial double categories and globularily generated pieces of double categories of bordisms, algebras, and von Neumann algebras, are all of length 1.

\

\noindent \textit{Free globularily generated double categories}

\

\noindent We will, from now on, identify every decorated bicategory $(\mathcal{B}^*,\mathcal{B})$ with the bicategory $\mathcal{B}$. The free globularily generated double category construction associates to every decorated bicategory $\mathcal{B}$ a globularily generated double category $Q_\mathcal{B}$ in such a way that $Q_\mathcal{B}$ fixes the data of $\mathcal{B}$ and such that the only relations satisfied by the squares of $Q_\mathcal{B}$ are those relations coming from relations satisfied by the 2-cells of $\mathcal{B}$ and the morphisms of $\mathcal{B}^*$. 

The intuitive idea behind the free globularily generated double category construction is as follows: Suppose we are provided with a decorated bicategory $\mathcal{B}$. We wish to construct, from the data of $\mathcal{B}$ alone, a double category $C$ satisfying the equation $H^*C=\mathcal{B}$, and we wish to do this in a minimal way. As outlined above we thus wish to construct a globularily generated double category $C$ satisfying the equation $H^*C=\mathcal{B}$. Such double category has $\mathcal{B}^*$ as category of objects, has the collection of 1-cells $\mathcal{B}_1$ of $\mathcal{B}$ as collection of horizontal morphisms, and all its squares can be expressed as a finite sequence of vertical and horizontal compositions of squares of the form:

\begin{center}

\begin{tikzpicture}
\matrix(m)[matrix of math nodes, row sep=4em, column sep=4em,text height=1.5ex, text depth=0.25ex]
{\bullet&\bullet&\bullet&\bullet\\
\bullet&\bullet&\bullet&\bullet\\};
\path[->,font=\scriptsize]
(m-1-1) edge node [above]{$\alpha$} (m-1-2)
        edge[blue] node [black][left]{$id$} (m-2-1)
(m-2-1) edge node [below]{$\beta$} (m-2-2)			
(m-1-2) edge [blue] node[black][right] {$id$} (m-2-2)
(m-1-1) edge [white] node[black][fill=white]{$\varphi$} (m-2-2)

(m-1-3) edge[red] node [black][above]{$id$} (m-1-4)
        edge node [left]{$f$} (m-2-3)
(m-2-3) edge [red]node[black][below]{$id$} (m-2-4)			
(m-1-4) edge node [right]{$f$} (m-2-4)
(m-1-3) edge [white] node[black][fill=white]{$i_f$} (m-2-4);
\end{tikzpicture}

\end{center}

\noindent with $\varphi$ being a 2-cell in $\mathcal{B}$ and $f$ being a morphism in $\mathcal{B}^*$. Moreover, the vertical filtration of $C$ provides a way to organize these expressions into strata indicating some measure of complexity.

The free globularily generated double category construction formally reproduces the above situation. We begin by formally associating to every vertical morphism $f:a\to b$ in the decoration $\mathcal{B}^*$ of the decorated bicategory we are provided with, a square of the form:

\begin{center}

\begin{tikzpicture}
\matrix(m)[matrix of math nodes, row sep=4em, column sep=4em,text height=1.5ex, text depth=0.25ex]
{a&a\\
b&b\\};
\path[->,font=\scriptsize]

(m-1-1) edge[red] node [black][above]{$id$} (m-1-2)
        edge node [left]{$f$} (m-2-1)
(m-2-1) edge [red]node[black][below]{$id$} (m-2-2)			
(m-1-2) edge node [right]{$f$} (m-2-2)
(m-1-1) edge [white] node[black][fill=white]{$i_f$} (m-2-2);
\end{tikzpicture}

\end{center}

\noindent and we formally associate to every 2-cell in $\mathcal{B}$ of the form:

\begin{center}

\begin{tikzpicture}
\matrix(m)[matrix of math nodes, row sep=4em, column sep=4em,text height=1.5ex, text depth=0.25ex]
{a&b\\};
\path[->,font=\scriptsize]
(m-1-1) edge [bend left=45] node [above]{$\alpha$}(m-1-2)
        edge [bend right=45] node [below]{$\beta$}(m-1-2)
        edge [white]node[black][fill=white]{$\varphi$}(m-1-2) ;
\end{tikzpicture}
\end{center}

\noindent a square of the form:

\begin{center}

\begin{tikzpicture}
\matrix(m)[matrix of math nodes, row sep=4em, column sep=4em,text height=1.5ex, text depth=0.25ex]
{a&b\\
 a&b\\};
\path[->,font=\scriptsize]

(m-1-1) edge node [above]{$\alpha$} (m-1-2)
        edge [blue]node [black][left]{$id$} (m-2-1)
(m-2-1) edge node [below]{$\beta$} (m-2-2)			
(m-1-2) edge [blue]node [black][right]{$id$} (m-2-2)
(m-1-1) edge [white] node[black][fill=white]{$\varphi$} (m-2-2);
\end{tikzpicture}

\end{center}

\noindent Having done this, we consider the path category generated by these squares. These are squares of length 1. We will write $F_1$ for this category. Inductively we define a category $F_k$ as the path category of the collection of formal horizontal compositions of squares in $F_{k-1}$, assuming $F_{k-1}$ has been defined. This provides a collection of squares with free horizontal and vertical composition rules. Dividing by a suitable relation we would obtain a free double category in the sense of \cite{DawsonPareFree}. Since we wish to fix the data provided by $\mathcal{B}$ we divide the structure we obtain by a finer equivalence relation $R_\infty$ and thus obtain a double category $Q_\mathcal{B}$ such that $Q_\mathcal{B}$ is globularily generated. If we choose $R_\infty$ carefully enough, the category of objects of $Q_\mathcal{B}$ will be $\mathcal{B}^*$ and the collection of horizontal morphisms of $Q_\mathcal{B}$ will be the collection of 1-cells of $\mathcal{B}$. 

We prove that thus defined the globularily generated double category $Q_\mathcal{B}$ associated to a decorated bicategory $\mathcal{B}$ not-necessarily provides a solution to problem \ref{prob} for $\mathcal{B}$. The only obstruction for this is that through composition operations in $Q_\mathcal{B}$ we may inadvertedly construct new globular squares not already in $\mathcal{B}$. We provide conditions on decorated bicategories $\mathcal{B}$ that guarantee that $Q_\mathcal{B}$ provides solutions to problem \ref{prob} for $\mathcal{B}$. Moreover, in the case in which $Q_\mathcal{B}$ does not provide solutions to problem \ref{prob} for $\mathcal{B}$ we prove that if we modify $\mathcal{B}$ enough, we can construct a decorated bicategory for which the free globularily generated double category does provide solutions to problem \ref{prob}. 

The free globularily generated double category construction provides a method for explicitly constructing double categories satisfying certain conditions. We use the free globularily generated double category construction to provide examples of double categories with non-trivial length and examples of double categories with infinite length. We relate the free globularily generated double category construction with the free product of groups and monoids, and with the free double category construction of \cite{DawsonPareFree}. Further, we apply free globularily generated double categories to provide formal solutions to the problem of existence of functorial extensions of both the Haagerup standard form and the Connes fusion operation.

Finally, in the second part of the present series of papers \cite{yo2} we provide an interpretation of the globularily generated double category construction as the object function of a functor $Q$ from \textbf{bCat}$^*$ to \textbf{gCat} satisfying the equation:

\[Q\dashv H^*\restriction_{\mbox{\textbf{gCat}}}\]

\noindent thus completing the diagram expressing the fact that $H^*$ factors through fibers of \textbf{dCat} modulo the globularily generated piece fibration $\gamma$ presented above. The restriction $H^*\restriction_{\mbox{\textbf{gCat}}}$ is faithful and thus the above result provides in particular the globularily generated double category construction with the structure of a free object in \textbf{gCat}. We will thus interpret free globularily generated double categores as sets of generators for bases mod $\gamma$ of solutions to problem \ref{prob}.

\

\noindent \textit{Conventions}

\

\noindent We follow the usual conventions for the theory of bicategories and double categories, with a few exceptions. The word double category will always mean pseudo double category. We will write $\mathcal{B}_0,\mathcal{B}_1$, and $\mathcal{B}_2$ for the collections of 0-, 1-, and 2-cells of a bicategory $\mathcal{B}$ and we will write $C_0,C_1$ for the category of objects and the category of squares of a double cateogory $C$. We will write horizontal identities and compositions as $i$ and $\ast$. We will write vertical compositions as word concatenation. We will write $\lambda,\rho$ and $A$ for left and right identity transformations and associators of both bicategories and double categories. As above, we will identify decorated bicategories with their underlying bicategories. Thus when we say that $\mathcal{B}$ is a decorated bicategory the letter $\mathcal{B}$ will denote both a decorated bicategory and its underlying bicategory. We will write $\mathcal{B}^*$ for the decoration of a decorated bicategory $\mathcal{B}$. For most of the paper we will interpret decorated bicategories as decorated horizontalizations of double categories, we will thus sometimes call the 0-. 1-, and 2-cells of the underlying bicategory of a decorated bicategory $\mathcal{B}$ the objects, the horizontal morphisms and the globular squares of $\mathcal{B}$ and we will call the morphisms of the decoration $\mathcal{B}^*$ of $\mathcal{B}$ the vertical morphisms of $\mathcal{B}$. Pictorially we will represent vertical identity endomorphisms by blue arrows and horizontal identity endomorphisms by red arrows as was done above. Squares of the form:

\begin{center}

\begin{tikzpicture}
\matrix(m)[matrix of math nodes, row sep=4em, column sep=4em,text height=1.5ex, text depth=0.25ex]
{\bullet&\bullet\\
\bullet&\bullet\\};
\path[->,font=\scriptsize]
(m-1-1) edge node [above]{$\alpha$} (m-1-2)
        edge[blue] node {} (m-2-1)
(m-2-1) edge node[below]{$\beta$} (m-2-2)			
(m-1-2) edge [blue]node{} (m-2-2)
(m-1-1) edge [white]node[black][fill=white]{$\varphi$}(m-2-2);
\end{tikzpicture}

\end{center}

\noindent will thus represent globular squares, squares of the form:

\begin{center}

\begin{tikzpicture}
\matrix(m)[matrix of math nodes, row sep=4em, column sep=4em,text height=1.5ex, text depth=0.25ex]
{\bullet&\bullet\\
\bullet&\bullet\\};
\path[->,font=\scriptsize]
(m-1-1) edge [red]node{} (m-1-2)
        edge node [left]{$f$} (m-2-1)
(m-2-1) edge [red]node{} (m-2-2)			
(m-1-2) edge node[right]{$f$} (m-2-2)
(m-1-1) edge [white]node[black][fill=white]{$i_f$}(m-2-2);
\end{tikzpicture}

\end{center}

\noindent will represent horizontal identities and squares of the form:

\begin{center}

\begin{tikzpicture}
\matrix(m)[matrix of math nodes, row sep=4em, column sep=4em,text height=1.5ex, text depth=0.25ex]
{a&a&b&b&c&c\\
a&a&b&b&c&c\\};
\path[->,font=\scriptsize]
(m-1-1) edge [red] node [above]{} (m-1-2)
        edge[blue] node [black][left]{$id$} (m-2-1)
(m-2-1) edge node [below]{$\beta$} (m-2-2)			
(m-1-2) edge [blue] node[black][right] {$id$} (m-2-2)
(m-1-1) edge [white] node[black][fill=white]{$\varphi$} (m-2-2)

(m-1-3) edge node [above]{$\alpha$} (m-1-4)
        edge[blue] node [left]{} (m-2-3)
(m-2-3) edge [red]node{} (m-2-4)			
(m-1-4) edge [blue]node [right]{} (m-2-4)
(m-1-3) edge [white] node[black][fill=white]{$\psi$} (m-2-4)

(m-1-5) edge [red]node {} (m-1-6)
        edge[blue] node [left]{} (m-2-5)
(m-2-5) edge [red]node{} (m-2-6)			
(m-1-6) edge [blue]node [right]{} (m-2-6)
(m-1-5) edge [white] node[black][fill=white]{$\eta$} (m-2-6);
\end{tikzpicture}

\end{center}

\noindent will represent a globular square from the horizontal identity $i_a$ of the object $a$ to the horizontal morphism $\beta$, a globular square from the horizontal morphism $\alpha$ to $i_b$ and a globular endomorphism of $i_c$ respectively.

\

\noindent \textit{Contents of the paper}

\

\noindent We now sketch the contents of the paper. In section \ref{s2} we present a detailed construction of the free globularily generated double category associated to a decorated bicategory. We do this in several steps and our construction not only yields a free globularily generated double category but a free vertical filtration which will allow us to associate numerical invariants to decorated bicategories. In section \ref{s3} we study relations between the free globularily generated double category construction and problem \ref{prob}. We provide conditions on decorated bicategories that ensure that the corresponding free globularily generated double category is an internalization and in situations in which this is not the case we introduce a method under which one can always extend a decorated bicategory to a decorated bicategory for which the free globularily generated double category is an internalizations. In section \ref{s4} we apply the free globularily generated double category construction to provide examples of double cateories with non-trivial length. In section \ref{s5} we study the free globularily generated double category in the case of deloopings of monoidal categories decorated by deloopings of groups. We prove that the free globularily generated double category associated to decorated bicategories of this type are always of length 1. Finally, in section \ref{s6} we apply a modification of the free globularily generated double category to provide compatible formal functorial extensions of the Haagerup standard form construction and the Connes fusion operation to certain linear categories of Hilbert spaces.

\section{The free globularily generated double category}\label{s2}

\noindent In this section we introduce the free globularily generated double category construction. The free globularily generated double category construction canonically associates a globularily generated double category to every decorated bicategory. The strategy behind the construction is to emulate the internal structure defined by the vertical and horizontal filtrations in abstract globularily generated double categories in order to obtain, from the data of a decorated bicategory alone, a globularily generated double category. The construction of the free globularily generated double category is rather involved and we divide it in several steps. We begin with a few preliminary definitions and results.

\

\noindent \textit{Preliminaries: Evaluations}

\

\noindent Let $X$ and $Y$ be sets. Let $s,t:X\to Y$ be functions. Let $x_1,...,x_n$ be a sequence in $X$. We will say that $x_1,...,x_k$ is compatible with respect to $s$ and $t$ if the equation $tx_{i+1}=sx_{i}$ holds for every $1\leq i\leq k-1$. Equivalently, $x_1,...,x_k$ is compatible with respect to $s$ and $t$ if $x_1,...,x_n$ is a composable sequence of morphisms in the free category generated by $X$ with $s$ and $t$ as domain and codomain functions respectively. Given a compatible sequence $x_1,...,x_k$ in $X$, we call any way of writing the word $x_k...x_1$ following an admissible parenthesis pattern, an evaluation of $x_1,...,x_k$. Equivalently, the evaluations of a compatible sequence $x_1,...,x_k$ are different ways of writing the word $x_k...x_1$ composing elements of $x_1,...,x_k$ two by two in the free category generated by $X$, with $s,t$ as domain and codomain functions. For example, $(yx)$ is the only evaluation of the two term compatible sequence $x,y$ and $(x(yz)),((xy)z)$ are the two evaluations of the compatible three term sequence $x,y,z$. We will write $X_{s,t}$ for the set of evaluations of finite sequences of elements of $X$, compatible with respect to $s$ and $t$.

Given functions $s,t:X\to Y$, we write $\tilde{s}$ and $\tilde{t}$ for the functions $\tilde{s},\tilde{t}:X_{s,t}\to Y$ defined as follows: Given an evaluation $\Phi$ of a compatible sequence $x_1,...,x_k$ in $X$ we make $\tilde{s}\Phi$ and $\tilde{t}\Phi$ to be equal to $sx_1$ and $tx_k$ respectively. Observe that the values $\tilde{s}\Phi$ and $\tilde{t}\Phi$ do not depend on the particular evaluation $\Phi$ of $x_1,...,x_k$. Given a pair of compatible sequences $x_1,...,x_k$ and $x_{k+1},...,x_n$ in $X$, such that the 2 term sequence $x_k,x_{k+1}$ is compatible, and given evaluations $\Phi$ and $\Psi$ of $x_1,...,x_k$ and $x_{k+1},...,x_n$, the equation $\tilde{t}\Phi=\tilde{s}\Psi$ is satisfied and the concatenation of $\Psi$ and $\Phi$ defines an evaluation of the sequence $x_1,...,x_k,x_{k+1},...,x_n$. We denote the concatenation of $\Phi$ and $\Psi$ satisfying the conditions above by $\Psi\ast_{s,t}\Phi$. This operation defines a function from $X_{s,t}\times_YX_{s,t}$ to $ X_{s,t}$ where the fibration in $X_{s,t}\times_YX_{s,t}$ is taken with respect to the pair $\tilde{s},\tilde{t}$. We write $\ast_{s,t}$ for this function.

Now, given sets $X,X',Y$, and $Y'$, and functions $s,t:X\to Y$ and $s',t':X'\to Y'$, we say that a pair of functions $\varphi:X\to X'$ and $\phi:Y\to Y'$ is compatible if the following two squares commute 

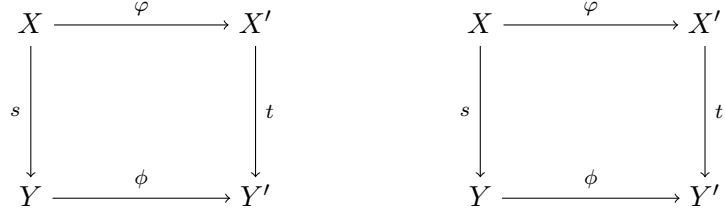
\begin{center}

\begin{tikzpicture}
\matrix(m)[matrix of math nodes, row sep=4em, column sep=4em,text height=1.5ex, text depth=0.25ex]
{X&X'&X&X'\\
Y&Y'&Y&Y'\\};
\path[->,font=\scriptsize]
(m-1-1) edge node[auto] {$\varphi$} (m-1-2)
        edge node[left] {$s$} (m-2-1)
(m-2-1) edge node[auto]	{$\phi$} (m-2-2)			
(m-1-2) edge node[auto] {$t$} (m-2-2)
(m-1-3) edge node[auto] {$\varphi$} (m-1-4)
        edge node[left] {$s$} (m-2-3)
(m-1-4) edge node[auto]	{$t$} (m-2-4)			
(m-2-3) edge node[auto] {$\phi$} (m-2-4);
\end{tikzpicture}

\end{center}

\noindent Given sets $X,X',Y$, and $Y'$, functions $s,t:X\to Y$ and $s',t':X'\to Y'$, and a compatible pair of functions $\varphi:X\to X'$ and $\phi:Y\to Y'$, if $x_1,...,x_k$ is a sequence in $X$, compatible with respect to $s,t$ then the sequence $\varphi x_1,...,\varphi x_k$ is compatible with respect to $s',t'$. Moreover, given an evaluation $\Phi$ of a compatible sequence $x_1,...,x_k$, the same parenthesis pattern defining the evaluation $\Phi$ defines an evaluation of the compatible sequence $\varphi x_1,...,\varphi x_k$. We write $\mu_{\varphi,\phi}\Phi$ for this evaluation. We write $\mu_{\varphi,\phi}$ for the function from $X_{s,t}$ to $X'_{s',t'}$ associating the evaluation $\mu_{\varphi,\phi}\Phi$ to every evaluation $\Phi$ in $X_{s,t}$. The proof of the following lemma is straightforward.

\begin{lemma}\label{Evaluations}
Let $X,X',Y$, and $Y'$ be sets. Let $s,t:X\to Y$ and $s',t':X'\to Y'$ be functions. Let $\varphi:X\to X'$ and $\phi:Y\to Y'$ be functions such that the pair $\varphi,\phi$ is compatible. In that case the function $\mu_{\varphi,\phi}$ associated to the pair $\varphi,\phi$ satisfies the following conditions

\begin{enumerate}
\item The following two squares commute

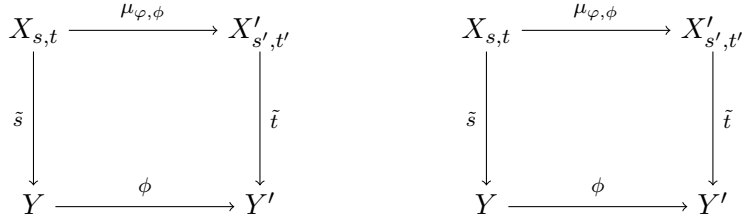
\begin{center}

\begin{tikzpicture}
\matrix(m)[matrix of math nodes, row sep=4em, column sep=4em,text height=1.5ex, text depth=0.25ex]
{X_{s,t}&X'_{s',t'}&X_{s,t}&X'_{s',t'}\\
Y&Y'&Y&Y'\\};
\path[->,font=\scriptsize]
(m-1-1) edge node[auto] {$\mu_{\varphi,\phi}$} (m-1-2)
        edge node[left] {$\tilde{s}$} (m-2-1)
(m-2-1) edge node[auto]	{$\phi$} (m-2-2)			
(m-1-2) edge node[auto] {$\tilde{t}$} (m-2-2)
(m-1-3) edge node[auto] {$\mu_{\varphi,\phi}$} (m-1-4)
        edge node[left] {$\tilde{s}$} (m-2-3)
(m-1-4) edge node[auto]	{$\tilde{t}$} (m-2-4)			
(m-2-3) edge node[auto] {$\phi$} (m-2-4);
\end{tikzpicture}

\end{center}

\item The following square commutes

\begin{center}

\begin{tikzpicture}
\matrix(m)[matrix of math nodes, row sep=4em, column sep=4em,text height=1.5ex, text depth=0.25ex]
{X_{s,t}\times_YX_{s,t}&X'_{s',t'}\times_Y X'_{s',t'}\\
X_{s,t}&X'_{s',t'}\\};
\path[->,font=\scriptsize]
(m-1-1) edge node[auto] {$\mu_{\varphi,\phi}\times_\phi\mu_{\varphi,\phi}$} (m-1-2)
        edge node[left] {$\ast_{s,t}$} (m-2-1)
(m-2-1) edge node[auto]	{$\mu_{\varphi,\phi}$} (m-2-2)			
(m-1-2) edge node[auto] {$\ast_{s',t'}$} (m-2-2);
\end{tikzpicture}

\end{center}

\end{enumerate}
\end{lemma}

\noindent \textit{Preliminaries: Notational conventions}

\

\noindent Let $\mathcal{B}$ be a decorated bicategory. Let $\alpha$ be a vertical morphism in $\mathcal{B}$. We will write $i_\alpha$ for the singleton $\left\{\alpha\right\}$. We call $i_\alpha$ the formal horizontal identity of $\alpha$. We write $\mathbb{G}$ for the union of the collection of globular squares in $\mathcal{B}$ and the collection of formal horizontal identities of vertical morphisms of $\mathcal{B}$. We will adopt the following notational conventions for the elements of $\mathbb{G}$.

\begin{enumerate}

\item Let $\Phi$ be a globular square in $\mathcal{B}$. We write $d_0\Phi,c_0\Phi$ for the domain and codomain of $\Phi$ in $\mathcal{B}$. Let $\alpha$ be a vertical morphism in $\mathcal{B}$. Let $a$ and $b$ be the domain and the codomain, in $\mathcal{B}^*$ of $\alpha$. We write $d_0i_\alpha$ and $c_0i_\alpha$ for the horizontal identities $id_a$ and $id_b$ of $a$ and $b$ in $\mathcal{B}$. We write $d_0$ and $c_0$ for the functions from $\mathbb{G}$ to $\mathcal{B}_1$ associating $d_0\Phi$ and $c_0\Phi$ to every element $\Phi$ of $\mathbb{G}$.

\item Let $\Phi$ be a globular square in $\mathcal{B}$. We write $s_0\Phi$ and $t_0\Phi$ for the source and target of $\Phi$ in $\mathcal{B}$. Let $\alpha$ be a vertical morphism in $\mathcal{B}$. In that case we write $s_0i_\alpha$ and $t_0i_\alpha$ for the morphism $\alpha$. We write $s_0$ and $t_0$ for the functions from $\mathbb{G}$ to the collection of vertical morphisms of $\mathcal{B}$ associating $s_0\Phi$ and $t_0\Phi$ to every element $\Phi$ of $\mathbb{G}$.
\end{enumerate}

\noindent The functions $d_0,c_0,s_0$, and $t_0$ defined above are easily seen to be related by the following  conditions

\begin{enumerate}
\item The following two triangles commute

\begin{center}

\begin{tikzpicture}
\matrix(m)[matrix of math nodes, row sep=4em, column sep=4em,text height=1.5ex, text depth=0.25ex]
{\mathbb{G}&\mathcal{B}_1&\mathbb{G}&\mathcal{B}_1\\
&\mathcal{B}_0&&\mathcal{B}_0\\};
\path[->,font=\scriptsize]
(m-1-1) edge node[auto] {$d_0$} (m-1-2)
        edge node[left] {$s_0$} (m-2-2)			
(m-1-2) edge node[auto] {$d$} (m-2-2)
(m-1-3) edge node[auto] {$d_0$} (m-1-4)
(m-1-4) edge node[auto]	{$t$} (m-2-4)			
(m-1-3) edge node[left] {$t_0$} (m-2-4);
\end{tikzpicture}

\end{center}

\item The following two triangles commute

\begin{center}

\begin{tikzpicture}
\matrix(m)[matrix of math nodes, row sep=4em, column sep=4em,text height=1.5ex, text depth=0.25ex]
{\mathbb{G}&\mathcal{B}_1&\mathbb{G}&\mathcal{B}_1\\
&\mathcal{B}_0&&\mathcal{B}_0\\};
\path[->,font=\scriptsize]
(m-1-1) edge node[auto] {$t_0$} (m-1-2)
        edge node[left] {$s_0$} (m-2-2)			
(m-1-2) edge node[auto] {$d$} (m-2-2)
(m-1-3) edge node[auto] {$t_0$} (m-1-4)
(m-1-4) edge node[auto]	{$t$} (m-2-4)			
(m-1-3) edge node[left] {$t_0$} (m-2-4);
\end{tikzpicture}

\end{center}
\end{enumerate}

\noindent Given a decorated bicategory $\mathcal{B}$, we denote by $p$ the function from the collection of evaluations $\mathcal{B}_{1_{d,c}}$ of $\mathcal{B}_1$, with respect to the pair formed by the domain and the codomain functions $d,c$ in the bicategory underlying  $\mathcal{B}$, to the set of horizontal morphisms $\mathcal{B}_1$ of $\mathcal{B}$ defined as follows: for every composable sequence $f_1,...,f_k$ of horizontal morphisms in $\mathcal{B}$ and for every evaluation $\Phi$ of $f_1,...,f_k$, the image $p\Phi$ of $\Phi$ under $p$ is equal to the horizontal composition of the sequence $f_1,...,f_k$ in $\mathcal{B}$ following the parenthesis pattern defining $\Phi$. We call $p$ the projection associated to $\mathcal{B}_1$.

\

\noindent \textit{The main construction: Inductive step}

\

\noindent Given a decorated bicategory $\mathcal{B}$ we will write $E_1$ for the collection of evaluations $\mathbb{G}_{s_0,t_0}$ of $\mathbb{G}$ with respect to the pair of functions $s_0,t_0$. We denote by $s_1$ and $t_1$ the functions $\tilde{s}_0$ and $\tilde{t}_0$. Thus defined $s_1$ and $t_1$ are functions from $E_1$ to Hom$_{\mathcal{B}^*}$. We write $\ast_1$ for the operation $\ast_{s_1,t_1}$ on $E_1$. Finally, we write $d_1$ for the composition $p\mu_{d_0,id}$, and we write $c_1$ for the composition $p\mu_{c_0,id}$. Thus defined $d_1$ and $c_1$ are functions from $E_1$ to the set of horizontal morphisms $\mathcal{B}_1$ of $\mathcal{B}$. The following theorem is the first step towards the free globularily generated double category construction.

\begin{thm}\label{InductiveStep}
Let $\mathcal{B}$ be a decorated bicategory. There exists a pair of sequences of triples $(E_k,d_k,c_k)$ and $(F_k,s_{k+1},t_{k+1})$, such that for each $k$, $E_k$ is a set containing $E_1$, $d_k,c_k$ are functions from $E_k$ to $\mathcal{B}_1$ extending the functions $d_1$ and $c_1$ defined above, $F_k$ is a category having $\mathcal{B}_1$ as collection of objects, and $s_{k+1},t_{k+1}$ are functors from $F_k$ to $\mathcal{B}^*$. The pair of sequences $(E_k,d_k,c_k)$ and $(F_k,s_{k+1},t_{k+1})$ satisfies the following conditions:

\begin{enumerate}
\item For every $k$, Hom$_{F_k}$ is contained in $E_{k+1}$. Moreover, $E_{k+1}$ is equal to the set of evaluations Hom$_{F_k s_{k+1},t_{k+1}}$ of Hom$_{F_k}$ with respect to the pair formed by the morphism functions of $s_{k+1}$ and $t_{k+1}$. 

\item For every $k$, $E_k$ is contained in Hom$_{F_k}$. Moreover, $F_k$ is equal to the free category generated by $E_k$ with functions $d_k$ and $c_k$ as domain and codomain functions respectively. The restriction of the morphism functions of $s_{k+1}$ and $t_{k+1}$ to $E_1$ are equal to the functions $s_1$ and $t_1$ defined above.

\item For every positive integer $k$ the following triangles commute:

\begin{center}

\begin{tikzpicture}
\matrix(m)[matrix of math nodes, row sep=4em, column sep=4em,text height=1.5ex, text depth=0.25ex]
{E_k&\mathcal{B}_1&E_k&\mathcal{B}_1\\
&\mathcal{B}_0&&\mathcal{B}_0\\};
\path[->,font=\scriptsize]
(m-1-1) edge node[auto] {$d_k,c_k$} (m-1-2)
        edge node[left] {$s_{k+1}$} (m-2-2)			
(m-1-2) edge node[auto] {$dom$} (m-2-2)
(m-1-3) edge node[auto] {$d_k,c_k$} (m-1-4)
(m-1-4) edge node[auto]	{$codom$} (m-2-4)			
(m-1-3) edge node[left] {$t_{k+1}$} (m-2-4);
\end{tikzpicture}

\end{center}
\end{enumerate}

\noindent The conditions 1-3 above determine the pair of sequences of triples $(E_k,d_k,c_k)$ and $(F_k,s_{k+1},t_{k+1})$

\end{thm}

\begin{proof}

Let $\mathcal{B}$ be a decorated bicategory. We wish to construct a pair of sequences of triples $(E_k,d_k,c_k)$ and $(F_k,s_{k+1},t_{k+1})$ with $k$ running through the collection of all positive integers, such that for each positive integer $k$, $E_k$ is a set extending to the set $E_1$ associated to $\mathcal{B}$, such that $d_k$ and $c_k$ are functions from $E_k$ to the collection $\mathcal{B}_1$ of horizontal morphisms of $\mathcal{B}$ extending the functions $d_1$ and $c_1$, such that $F_k$ is a category having $\mathcal{B}_1$ as set of objects and $s_{k+1},t_{k+1}$ are functors from $F_k$ to $\mathcal{B}^*$. Moreover, we wish to define the pair of sequences $(E_k,d_k,c_k)$ and $(F_k,s_{k+1},t_{k+1})$ in such a way that conditions 1-3 above are satisfied.

We proceed by induction on $k$. We begin by defining the triple $(F_1,s_2,t_2)$. We make $F_1$ to be the free category generated by $E_1$ with respect to $d_1,c_1$. The functions $s_1$ and $t_1$ are compatible with $d_1$ and $c_1$ and thus admit a unique extension to functors from $F_1$ to the decoration $\mathcal{B}^*$ of $\mathcal{B}$. We make $s_2$ and $t_2$ be the corresponding functorial extensions of $s_1$ and $t_1$. We write $\bullet_1$ for the composition operation in $F_1$. 

Let now $k$ be a positive integer strictly greater than 1. Assume that we have extended the definition of the triple $(E_1,d_1,c_1)$ and the definition of the triple $(F_1,s_2,t_2)$ of the previous paragraph, to a sequence of pairs of triples $(E_m,d_m,c_m)$ and $(F_m,s_{m+1},t_{m+1})$ for every $m\leq k$, where $E_m$ is assumed to be a set containing $E_1$, $d_m$ and $c_m$ are assumed to be functions from $E_m$ to $\mathcal{B}_1$ extending $d_1$ and $c_1$, $F_m$ is assumed to be a category containing $F_1$ as subcategory, and $s_{m+1}$ and $t_{m+1}$ are assumed to be functors from $F_m$ to $\mathcal{B}^*$ extending $s_2$ and $t_2$ respectively. Moreover, we assume that the pair of sequences of triples $(E_m,d_m,c_m)$ and $(F_m,s_{m+1},t_{m+1})$ satisfies conditions 1-3 above.

We now wish to extend the definition of the pair of sequences of triples $(E_m,d_m,c_m)$ and $(F_m,s_{m+1},t_{m+1})$ to the definition of a pair $(E_k,d_k,c_k)$ and $(F_k,s_{k+1},t_{k+1})$ satisfying the conditions of the theorem. We begin with the definition of the triple $(E_k,d_k,c_k)$. We make $E_k$ to be the collection of evaluations Hom$_{F_{k-1} s_k,t_k}$ of Hom$_{F_{k-1}}$ with respect to the pair formed by the morphism functions of $s_k$ and $t_k$. We write $s_{k+1}$ and $t_{k+1}$ for the extensions $\tilde{s}_k$ and $\tilde{t}_k$, to $E_k$, of the morphism functions of $s_k$ and $t_k$. We denote by $\ast_k$ the concatenation operation $\ast_{s_{k+1},t_{k+1}}$ in $E_k$, with respect to $s_{k+1}$ and $t_{k+1}$. We now make the function $d_k$ to be the composition $p\mu_{d_{k-1},id}$ of the function associated to the pair formed by the domain function $d_{k-1}$ in $F_{k-1}$ and the identity function in the collection of horizontal morphisms $\mathcal{B}_1$ of $\mathcal{B}$ and the projection $p$ associated to the collection of horizontal morphisms $\mathcal{B}_1$ of $\mathcal{B}$. We make $c_k$ to be the composition $p\mu_{c_{k-1},id}$ of the function associated to the pair formed by the codomain function in $F_{k-1}$ and the identity function in the collection of horizontal morphisms $\mathcal{B}_1$ of $\mathcal{B}$, and the projection $p$ associated to the collection of horizontal morphisms $\mathcal{B}_1$ of $\mathcal{B}$. The functions $d_k$ and $c_k$ are well defined. Thus defined $d_k$ and $c_k$ are functions from $E_k$ to $\mathcal{B}_1$ satisfying condition 3 of the theorem by lemma 3.1. We now define the triple $(F_k,s_{k+1},t_{k+1})$. We make the category $F_k$ to be the free category generated by $E_k$, with $d_k$ and $c_k$ as domain and codomain functions. The collection of objects of $F_k$ is thus $\mathcal{B}_1$. We write $\bullet_k$ for the composition operation in $F_k$. By the fact that the functions $d_k,c_k,s_k$ and $t_k$ satisfy the condition 3 of the theorem it follows that the pairs formed by $s_k$ and $t_k$ together with the domain and codomain functions defined on $\mathcal{B}_1$ admit unique extensions to functors from $F_k$ to $\mathcal{B}^*$. We make $s_{k+1}$ and $t_{k+1}$ to be these functors. Thus defined the triple $(F_k,s_{k+1},t_{k+1})$ satisfies the conditions of the theorem.

It is obvious that conditions 1-3 of the theorem determine the pair of sequences $(E_k,d_k,c_k)$ and $(F_k,s_{k+1},t_{k+1})$. This concludes the proof.
\end{proof}

\begin{obs}\label{ObsOnInductiveStep1}
Let $\mathcal{B}$ be a decorated bicategory. Let $m,k$ be positive integers such that $m$ is less than or equal to $k$. In that case, as defined above, $E_m$ is contained in $E_k$, $d_m$ and $c_m$ are equal to the restrictions to $E_k$ of $d_k$ and $c_k$ respectively, and the concatenation operation $\ast_m$ is equal to the restriction to $E_m$ of $\ast_k$. Moreover, $F_m$ is a subcategory of $F_k$ and the functors $s_{m+1},t_{m+1}$ are equal to the restrictions to $F_m$ of $s_{k+1}$ and $t_{k+1}$ respectively. 
\end{obs}

\begin{obs}\label{ObsOnInductiveStep2}
Let $\mathcal{B}$ be a decorated bicategory. Let $k$ be a positive integer. The following two squares commute where $q^k_1$ and $q^k_2$ denote the left and the right projections of $E_k\times_{\mbox{Hom}_{\mathcal{B}^*}}E_k$ onto $E_k$ respectively.

\begin{center}

\begin{tikzpicture}
\matrix(m)[matrix of math nodes, row sep=4em, column sep=3em,text height=1.5ex, text depth=0.25ex]
{E_k\times_{\mbox{Hom}_{\mathcal{B}^*}}E_k&E_k&E_k\times_{\mbox{Hom}_{\mathcal{B}^*}}E_k&E_k\\
E_k&\mbox{Hom}_{\mathcal{B}^*}&E_k&\mbox{Hom}_{\mathcal{B}^*}\\};
\path[->,font=\scriptsize]
(m-1-1) edge node[auto] {$\ast_k$}  (m-1-2)
        edge node[left] {$q^k_1$}   (m-2-1)			
(m-1-2) edge node[auto] {$s_{k}$} (m-2-2)
(m-2-1) edge node[below]{$s_{k}$} (m-2-2)
(m-1-3) edge node[auto] {$\ast_k$}  (m-1-4)
(m-1-4) edge node[auto]	{$t_{k}$} (m-2-4)			
(m-1-3) edge node[left] {$q^k_2$}   (m-2-3)
(m-2-3) edge node[below]{$s_{k}$} (m-2-4);
\end{tikzpicture}

\end{center}

\end{obs}

\

\noindent \textit{The main construction: Taking limits}

\

\noindent As the next step in the free globularily generated double category construction we apply a limiting procedure to the pieces of structure obtained in Theorem \ref{InductiveStep}. We define an equivalence relation $R_\infty$ such that after taking quotients modulo $R_\infty$ we will obtain the necessary relations defining a double category. We keep the notation from Theorem \ref{InductiveStep}.

\begin{notation}\label{LimitNotation}
Let $\mathcal{B}$ be a decorated bicategory. We write $E_\infty$ for $\bigcup_{k=1}^\infty E_k$. We write $d_\infty$ and $c_\infty$ for $\varinjlim d_k$ and $\varinjlim c_k$ respectively. Thus defined $d_\infty$ and $c_\infty$ are functions from $E_\infty$ to $\mathcal{B}_1$. We write $\ast_\infty$ for $\varinjlim \ast_k$. Thus defined $\ast_\infty$ is a function from $E_\infty\times_{\mathcal{B}^*}E_\infty$ to $E_\infty$. We write $F_\infty$ for the category $\varinjlim F_k$. The collection of objects of $F_\infty$ is the collection of horizontal morphisms $\mathcal{B}_1$ of $\mathcal{B}$, and the collection of morphisms of $F_\infty$ is $E_\infty$. The domain and codomain functions of $F_\infty$ are $d_\infty$ and $c_\infty$. We write $s_\infty,t_\infty$ for $\varinjlim s_k$ and $\varinjlim_k$ respectively. Thus defined $s_\infty$ and $t_\infty$ are functors from $F_\infty$ to $\mathcal{B}^*$. Finally, we write $\bullet_\infty$ for the composition operation of $F_\infty$. Thus defined $\bullet_\infty$ is equal to $\varinjlim \bullet_k$ where for every $k$ $\bullet_k$ is the composition operation of $F_k$.
\end{notation}

\begin{definition}\label{RInfinityDefinition}
Let $\mathcal{B}$ be a decorated bicategory. We write $R_\infty$ for the equivalence relation generated by the following relations defined on $E_\infty$:

\begin{enumerate}

\item Let $\Phi_i,\Psi_i$, $i=1,2$ be morphisms in $F_\infty$ such that the pairs $\Phi_1,\Phi_2$ and $\Psi_1,\Psi_2$ are compatible with respect to the pair $s_\infty,t_\infty$ and such that the pairs $\Phi_i,\Psi_i$, $i=1,2$ are both compatible with respect to the pair $d_\infty,c_\infty$. We identify the compositions:

\[(\Psi_2\ast_\infty\Psi_1)\bullet_\infty(\Phi_2\ast_\infty\Phi_1) \ \mbox{and} \ (\Psi_2\bullet_\infty\Phi_2)\ast_\infty(\Psi_1\bullet_\infty\Phi_1)\]

\item Let $\Phi$ and $\Psi$ be globular squares of $\mathcal{B}$ such that the pair $\Phi,\Psi$ is compatible with respect to $s_\infty,t_\infty$. We identify $\Psi\bullet_\infty\Phi$ with the vertical compostion $\Psi\Phi$ of $\Phi$ and $\Psi$ in $\mathcal{B}$. Moreover, if $\alpha$ and $\beta$ are vertical morphisms in $\mathcal{B}$ such that the pair $\alpha,\beta$ is composable in $\mathcal{B}^*$, then we identify $i_\beta\bullet_\infty i_\alpha$ with $i_{\beta\alpha}$.

\item Let $\Phi$ and $\Psi$ be globular squares of $\mathcal{B}$ such that the pair $\Phi,\Psi$ is compatible with respect to $s_\infty,t_\infty$. We identify $\Psi\ast_\infty\Phi$ with the horizontal composition $\Psi\ast\Phi$ in $\mathcal{B}$.

\item Let $\Phi$ be a morphism in $F_\infty$. We identify $\Phi$ with the compositions

\[\lambda_{c_\infty\Phi}\bullet_\infty(\Phi\ast_\infty i_{t_\infty\Phi})\bullet_\infty \lambda_{d_\infty\Phi}^{-1} \ \mbox{and} \ \rho_{c_\infty\Phi}\bullet_\infty(i_{s_\infty\Phi}\ast_\infty \Phi)\bullet_\infty \rho_{d_\infty\Phi}^{-1}\]

\noindent where $\lambda$ and $\rho$ denote the left and right identity transformations of the bicategory underlying $\mathcal{B}$.

\item Let $\Phi,\Psi,\Theta$ be elements of $E_\infty$ such that the triple $\Phi,\Psi,\Theta$ is compatible with respect to the pair $s_\infty,t_\infty$. In that case we identify the compositions:

\[A_{c_{\infty}\Phi,c_{\infty}\Psi,c_{\infty}\Theta}\bullet_\infty[\Theta\ast_\infty(\Psi\ast_\infty\Phi)] \ \mbox{and} \ [(\Theta\ast_\infty\Psi)\ast_\infty \Phi]\bullet_\infty A_{d_{\infty}\Phi,d_{\infty}\Psi,d_{\infty}\Theta}\]

\noindent where $A$ denotes the associator of the bicategory underlying $\mathcal{B}$.
\end{enumerate}

\end{definition}

\begin{lemma}\label{CompatibilityRInfinity}
Let $\mathcal{B}$ be a decorated bicategory. $R_\infty$ is compatible with the domain, codomain, and composition operation functions $d_\infty,c_\infty$, and $\bullet_\infty$ of $F_\infty$.
\end{lemma}

\begin{proof}
Let $\mathcal{B}$ be a decorated bicategory. We wish to prove that the equivalence relation $R_\infty$ is compatible with domain, codomain, and composition operation functions $d_\infty,c_\infty$, and $\bullet_\infty$ defining the category structure on $F_\infty$.

Let $\Phi_i,\Psi_i$, $i=1,2$ be morphisms in $F_\infty$ such that the pairs $\Phi_1,\Phi_2$ and $\Psi_1,\Psi_2$ are compatible with respect to source and target functors $s_\infty,t_\infty$ of $F_\infty$ and such that the pairs $\Phi_i,\Psi_i$, $i=1,2$ are compatible with respect to domain and codomain functions $d_\infty,c_\infty$. In that case the domain 

\[d_\infty (\Psi_1\ast_\infty\Psi_1)\bullet_\infty(\Phi_2\ast_\infty\Phi_1)\]

\noindent of $(\Psi_1\ast_\infty\Psi_1)\bullet_\infty(\Phi_2\ast_\infty\Phi_1)$ is equal to the domain $d_\infty(\Phi_2\ast_\infty\Phi_1)$ of $(\Phi_2\ast_\infty\Phi_1)$ which in turn is equal to the composition $d_\infty\Phi_2 d_\infty\Phi_1$ of the domains $d_\infty\Phi_1$ and $d_\infty\Phi_2$ of $\Phi_1$ and $\Psi_1$. Now, the domain 

\[d_\infty (\Psi_2\bullet_\infty\Phi_2)\ast_\infty(\Psi_1\bullet_\infty\Phi_1)\]

\noindent of the composition $(\Psi_1\bullet_\infty\Phi_1)\ast_\infty(\Psi_2\bullet_\infty\Phi_2)$ is equal to the composition 

\[d_\infty (\Psi_2\bullet_\infty\Phi_2)d_\infty(\Psi_1\bullet_\infty\Phi_1)\]

\noindent of the domain $d_\infty (\Psi_1\bullet_\infty\Phi_1)$ of the composition $\Psi_1\bullet_\infty\Phi_1$ and the domain $d_\infty(\Psi_2\bullet_\infty\Phi_2)$ of the composition $\Psi_2\bullet_\infty\Phi_2$. The domain $d_\infty (\Psi_1\bullet_\infty\Phi_1)$ of the composition $\Psi_1\bullet_\infty\Phi_1$ is equal to the domain $d_\infty\Phi_1$ of $\Phi_1$ and the domain $d_\infty(\Psi_2\bullet_\infty\Phi_2)$ of the composition $\Psi_2\bullet_\infty\Phi_2$ is equal to the domain $d_\infty \Phi_2$ of $\Phi_2$. Thus the domain 

\[d_\infty(\Psi_2\bullet_\infty\Phi_2)\ast_\infty(\Psi_1\bullet_\infty\Phi_1)\]

\noindent of the composition $(\Psi_2\bullet_\infty\Phi_2)\ast_\infty(\Psi_1\bullet_\infty\Phi_1)$ is equal to the composition $d_\infty d_\infty\Phi_2 d_\infty\Phi_1$ of the domain $d_\infty\Phi_1$ of $\Phi_1$ and the domain $d_\infty\Phi_2$ of $\Phi_2$. We conclude that equivalence relation 1 in the definition of $R_\infty$ is compatible with respect to the domain function $d_\infty$ of $F_\infty$. A similar computation proves that relation 1 in the definition of $R_\infty$ is compatible with respect to the codomain function $c_\infty$ of $F_\infty$.

Let now $\Phi,\Psi$ be elements of $\mathbb{G}$. suppose pair $\Phi,\Psi$ is compatible with respect to the domain and codomain functions $d_\infty$ and $c_\infty$ of $F_\infty$. In that case the domain and codomain $d_\infty \Psi\bullet_\infty\Phi$ and $\Psi\bullet_\infty\Phi$ of the composition $\Psi\bullet_\infty\Phi$ are equal to the domain $d_\infty\Phi$ of $\Phi$ and the codomain $c_\infty\Psi$ of $\Psi$ respectively. The domain and codomain $d_\infty \Psi\Phi$ and $c_\infty \Psi\Phi$ of the vertical composition $\Psi\Phi$ of $\Phi$ and $\Psi$ is equal to the domain $d_\infty\Phi$ of $\Phi$ and the codomain $c_\infty\Psi$ of $\Psi$. We conclude that relation 2 in the definition of $R_\infty$ is compatible with the domain and codomain functions $d_\infty$ and $c_\infty$ of $F_\infty$.

Let $\Phi$ and $\Psi$ be globular squares in $\mathcal{B}$. Suppose that the pair $\Phi,\Psi$ is compatible with respect to the morphism functions of functors $s_\infty$ and $t_\infty$. In that case the domain $d_\infty \Psi\ast_\infty\Phi$ and the codomain $c_\infty\Psi\ast_\infty \Phi$ of the horizontal composition $\Psi\ast_\infty\Phi$ of $\Phi$ and $\Psi$ are equal to the compositions, in $\mathcal{B}$, $d_\infty\Psi d_\infty\Phi$ and $c_\infty\Psi d_\infty\Phi$ respectively. Now, the domain $d_\infty\Psi\ast\Phi$ and the codomain $c_\infty\Psi\ast\Phi$ of the horizontal composition, in $\mathcal{B}$, of $\Phi$ and $\Psi$ is equal to the compositions $d_\infty\Phi d_\infty\Psi$ and $c_\infty\Phi c_\infty\Psi$ respectively. This proves that relation 3 in the definition of $R_\infty$ is compatible with the domain and codomain functions $d_\infty$ and $c_\infty$ of $F_\infty$.

Let now $\Phi$ be a general morphism in $F_\infty$. In that case the domain 

\[d_\infty \lambda_{c_\infty\Phi}\bullet_\infty(\Phi\ast_\infty i_{t_\infty\Phi})\bullet_\infty \lambda_{d_\infty\Phi}^{-1}\]

\noindent of the composition $\lambda_{c_\infty\Phi}\bullet_\infty(\Phi\ast_\infty i_{t_\infty\Phi})\bullet_\infty \lambda_{d_\infty\Phi}^{-1}$ is equal to the domain $d_\infty \lambda_{d_\infty\Phi}^{-1}$ of $\lambda_{d_\infty\Phi}^{-1}$, which is equal to the domain $d_\infty\Phi$ of $\Phi$. Similarly the domain 

\[d_\infty \rho_{c_\infty\Phi}\bullet_\infty(i_{s_\infty\Phi}\ast_\infty \Phi)\bullet_\infty \rho_{d_\infty\Phi}^{-1}\]

\noindent of the composition $\rho_{c_\infty\Phi}\bullet_\infty(i_{s_\infty\Phi}\ast_\infty \Phi)\bullet_\infty \rho_{d_\infty\Phi}^{-1}$ is equal to the domain $d_\infty \rho^{-1}_{d_\infty\Phi}$ of $\rho_{d_\infty\Phi}^{-1}$, which is equal to the domain $d_\infty \Phi$ of $\Phi$. We conclude, from this that relation 4 in the definition of $R_\infty$ is compatible with the domain function $d_\infty$ of $F_\infty$. An analogous computation proves that relation 4 in the definition of $R_\infty$ is compatible with the codomain function $c_\infty$ of $F_\infty$.

Let $\Phi,\Psi,\Theta$ be general morphisms in $F_\infty$. Suppose that the triple $\Phi,\Psi,\Theta$ is compatible with respect to the morphism functions of $s_\infty$ and $t_\infty$. In that case the domain 

\[d_\infty A_{c_{\infty}\Phi,c_{\infty}\Psi,c_{\infty}\Theta}\bullet_\infty[\Theta\ast_\infty(\Psi\ast_\infty\Phi)]\]

\noindent is equal to the composition $A_{c_{\infty}\Phi,c_{\infty}\Psi,c_{\infty}\Theta}\bullet_\infty[\Theta\ast_\infty(\Psi\ast_\infty\Phi)]$ is equal to domain $d_\infty \Theta\ast_\infty(\Psi\ast_\infty\Phi)$ of the composition $\Theta\ast_\infty(\Psi\ast_\infty\Phi)$ which in turn is equal to the composition $d_\infty\Theta d_\infty\Psi d_\infty\Phi$ in $\mathcal{B}$. Now, the domain

\[d_\infty [(\Theta\ast_\infty\Psi)\ast_\infty \Phi]\bullet_\infty A_{d_{\infty}\Phi,d_{\infty}\Psi,d_{\infty}\Theta}\]

\noindent is equal to the domain $d_\infty A_{d_{\infty}\Phi,d_{\infty}\Psi,d_{\infty}\Theta}$ of the associator $A_{d_{\infty}\Phi,d_{\infty}\Psi,d_{\infty}\Theta}$ associated to the triple $\Phi,\Psi,\Theta$, which is, by definition, equal to the composition $d_\infty \Theta d_\infty \Psi d_\infty \Phi$ in $\mathcal{B}$. We conclude that relation 5 in the definition of $R_\infty$ is compatible with the domain function $d_\infty$ of $F_\infty$. An analogous computation proves that relation 5 in the definition of $R_\infty$ is compatible with the codomain function $c_\infty$ of $F_\infty$.

Finally, the fact that the equivalence relation $R_\infty$ is compatible with respect to the composition function $\bullet_\infty$ on $F_\infty$ follows from the fact that $F_\infty$ is the limit of a sequence of free categories.
\end{proof}

\

\noindent \textit{The main construction: Dividing by $R_\infty$}

\

\noindent As the next step of the free globularily generated double category construction we divide the category $F_\infty$ defined in \ref{LimitNotation} by the equivalence relation $R_\infty$ and we prove that the structure thus obtained is compatible with the rest of the pieces of structure in Theorem \ref{InductiveStep}.

\

\begin{definition}\label{Dividingdefinition}
Let $\mathcal{B}$ be a decorated bicategory. We write $V_\infty$ for the quotient category $F_\infty/R_\infty$. We keep writing $d_\infty,c_\infty$, and $\bullet_\infty$ for the domain, codomain, and composition operation functions in $V_\infty$. We write $H_\infty$ for the collection of morphisms of $V_\infty$. Thus defined $H_\infty$ is equal to the quotient $E_\infty/R_\infty$ of the collection of morphisms $E_\infty$ of $F_\infty$ modulo $R_\infty$.
\end{definition}

\begin{lemma}\label{Dividinglemma}
Let $\mathcal{B}$ be a decorated bicategory. In that case the source and target functors $s_\infty$ and $t_\infty$, and the horizontal composition functor $\ast_\infty$ are all compatible with $R_\infty$. 
\end{lemma}

\begin{proof}
Let $\mathcal{B}$ be a decorated bicategory. We wish to prove that in that the source and target functors $s_\infty$ and $t_\infty$, and the horizontal composition functor $\ast_\infty$ defined on $F_\infty$ associated to $\mathcal{B}$ are compatible with $R_\infty$.

Let first $\Phi_i,\Psi_i$, $i=1,2$ be morphisms in $F_\infty$ such that the pairs $\Phi_i,\Psi_i$, $i=1,2$ are compatible with respect to the domain and codomain functions $d_\infty$ and $c_\infty$ in $F_\infty$ and such that the pairs $\Phi_1,\Phi_2$ and $\Psi_1,\Psi_1$ are compatible with respect to the source and target functors $s_\infty,t_\infty$ in $F_\infty$. In that case the source 

\[s_\infty (\Psi_2\ast_\infty\Psi_1)\bullet_\infty(\Phi_2\ast_\infty\Phi_1)\]

\noindent of composition $(\Psi_2\ast_\infty\Psi_1)\bullet_\infty(\Phi_2\ast_\infty\Phi_1)$ is equal to the composition

\[s_\infty (\Psi_2\ast_\infty\Psi_1)s_\infty(\Phi_2\ast_\infty\Phi_1)\]

\noindent in the decoration $\mathcal{B}^*$ of $\mathcal{B}$ of the vertical morphisms $s_\infty(\Phi_2\ast_\infty\Phi_1)$ and $s_\infty (\Psi_2\ast_\infty\Psi_1)$. The source $s_\infty(\Phi_2\ast_\infty\Phi_1)$ of the concatenation $\Phi_2\ast_\infty\Phi_1$ is equal to the source $s_\infty\Phi_1$ of $\Phi_1$ and the source $s_\infty (\Psi_2\ast_\infty\Psi_1)$ of the concatenation $\Psi_2\ast_\infty\Psi_1$ is equal to the source $s_\infty\Psi_2$. The source 

\[s_\infty (\Psi_2\ast_\infty\Psi_1)\bullet_\infty(\Phi_2\ast_\infty\Phi_1)\]

\noindent of composition $(\Psi_2\ast_\infty\Psi_1)\bullet_\infty(\Phi_2\ast_\infty\Phi_1)$ is thus equal to the composition $s_\infty\Psi_1s_\infty\Phi_1$ in $\mathcal{B}^*$. Now the source

\[s_\infty(\Psi_2\bullet_\infty\Phi_2)\ast_\infty(\Psi_1\bullet_\infty\Phi_1)\]

\noindent of the concatenation $(\Psi_2\bullet_\infty\Phi_2)\ast_\infty(\Psi_1\bullet_\infty\Phi_1)$ is equal to the source $s_\infty(\Psi_1\bullet_\infty\Phi_1)$ of composition $\Psi_1\bullet_\infty\Phi_1$. Now, the source $s_\infty(\Psi_1\bullet_\infty\Phi_1)$ is equal to the composition $s_\infty\Psi_1 s_\infty\Phi_1$ of the source $s_\infty\Phi_1$ and $s_\infty\Psi_1$ in $\mathcal{B}^*$. This proves that the source functor $s_\infty$ is compatible with respect to relation 1 in the definition of $R_\infty$. An analogous argument proves that the target functor $t_\infty$ in $F_\infty$ is compatible with respect to relation 1 in the definition of $R_\infty$.

Let now $\Phi$ and $\Psi$ be morphisms in $\mathbb{G}$ such that the pair $\Phi,\Psi$ is compatible with respect to the domain and codomain functions $d_\infty$ and $c_\infty$ in $F_\infty$. Suppose first that $\Phi$ and $\Psi$ are globular morphisms in $\mathcal{B}$ such that the domain in $\mathcal{B}$ of the domain and codomain of $\Phi$ and $\Psi$ in $\mathcal{B}$ respectively are equal to the object $a$ in $\mathcal{B}$. In that case the source $s_\infty\Psi\bullet_\infty\Phi$ of the composition $\Psi\bullet_\infty\Phi$ is equal to the composition $s_\infty\Psi s_\infty\Phi$ of $s_\infty\Phi$ and $s_\infty\Phi$ in $\mathcal{B}^*$. Now the source $s_\infty\Psi$ of $\Phi$ and the source $s_\infty\Psi$ of $\Psi$ are both equal to the identity endomorphism $id_a$ of the object $a$ in the decoration $\mathcal{B}^*$ of $\mathcal{B}$. Now the domain in $\mathcal{B}$ of the vertical composition $\Psi\bullet\Phi$ in $\mathcal{B}$ is equal to the domain of $\Phi$ in $\mathcal{B}$ and thus the domain of the domain in $\mathcal{B}$ of the vertical composition $\Psi\bullet\Phi$ of $\Phi$ and $\Psi$ is equal to the object $a$ of $\mathcal{B}$. It follows, from this, that the source $s_\infty\Psi\bullet\Phi$ of the globular morphism in $\mathcal{B}$ formed as the vertical composition $\Psi\bullet\Phi$ of $\Phi$ and $\Psi$ in $\mathcal{B}$ is equal to the identity endomorphism $id_a$ of the object $a$ in $\mathcal{B}$. The source functor $s_\infty$ in $F_\infty$ is thus compatible with the restriction to the collection of globular morphisms of $\mathcal{B}$ of relation 2 in the definition of $R_\infty$. Suppose now that the morphisms $\Phi$ and $\Psi$ are formal horizontal identities $i_\alpha$ and $i_\beta$ respectively, of a composable pair of vertical morphisms $\alpha,\beta$ in $\mathcal{B}$. In this case the source $s_\infty i_\beta\bullet_\infty i_\alpha$ of the vertical composition $i_\beta\bullet_\infty i_\alpha$ of $i_\alpha$ and $i_\beta$ is equal to the composition $s_\infty i_\beta s_\infty i_\alpha$ of $s_\infty i_\alpha$ and $i_\beta$ in $\mathcal{B}^*$. The source $s_\infty i_\alpha$ is equal to the morphism $\alpha$ and the source $s_\infty i_\beta$ of $i_\beta$ is equal to the morphism $\beta$. We conclude that the source $s_\infty i_\beta\bullet_\infty i_\alpha$ of the composition $i_\beta\bullet_\infty i_\alpha$ is equal to the composition $\beta\alpha$ of $\alpha$ and $\beta$ in $\mathcal{B}^*$. Finally, the source $s_\infty i_{\beta\alpha}$ of the formal horizontal identity $i_{\beta\alpha}$ of the composition $\beta\alpha$ is equal to the composition $\beta\alpha$. This proves that the source functor $s_\infty$ is compatible with the restriction to the collection of formal horizontal identities of relation 2 in the definition of $R_\infty$. We conclude that the source functor $s_\infty$ is compatible with respect to relation 2 in the definition of $R_\infty$. An analogous argument proves that the target functor $t_\infty$ in $F_\infty$ is compatible with respect to relation 2 in the definition of $R_\infty$.

Let now $\Phi$ and $\Psi$ be globular morphisms in $\mathcal{B}$ such that the pair $\Phi,\Psi$ is compatible with respect to the source and the target functors $s_\infty,t_\infty$ in $F_\infty$. The source $s_\infty\Psi\ast_\infty\Phi$ of the concatenation $\Psi\ast_\infty\Phi$ of $\Phi$ and $\Psi$ is equal to the source $s_\infty\Phi$ which is equal to the identity endomorphism in $\mathcal{B}^*$ of the domain of the domain in $\mathcal{B}$ of $\Phi$. Now, the source $s_\infty\Psi\ast\Phi$ of the globular morphism in $\mathcal{B}$ formed as the horizontal composition $\Psi\ast\Phi$ in $\mathcal{B}$ of $\Phi$ and $\Psi$ is equal to the identity endomorphism in $\mathcal{B}^*$ of the domain of the domain in $\mathcal{B}$ of the composition $\Psi\ast\Phi$. The domain of the domain of the horizontal composition $\Phi\ast\Phi$ is equal to the domain of the domain of $\Phi$. We conclude that the source $s_\infty\Psi\ast\Phi$ of the horizontal composition $\Psi\ast\Phi$ is equal to the identity endomorphism in $\mathcal{B}^*$ of the domain of the domain in $\mathcal{B}$ of $\Phi$. The source functor $s_\infty$ in $\mathcal{B}$ is thus compatible with relation 3 in the definition of $R_\infty$. An identical argument proves that the target functor $t_\infty$ in $F_\infty$ is compatible with relation 3 in the definition of $R_\infty$.

Let now $\Phi$ be a general morphism in $F_\infty$. In that case the source

\[s_\infty \lambda_{c_\infty\Phi}\bullet_\infty(\Phi\ast_\infty i_{t_\infty\Phi})\bullet_\infty \lambda_{d_\infty\Phi}^{-1}\]

\noindent of the composition $\lambda_{c_\infty\Phi}\bullet_\infty(\Phi\ast_\infty i_{t_\infty\Phi})\bullet_\infty \lambda_{d_\infty\Phi}^{-1}$ is equal to the composition

\[s_\infty \lambda_{c_\infty\Phi}s_\infty(\Phi\ast_\infty i_{t_\infty\Phi})s_\infty \lambda_{d_\infty\Phi}^{-1}\] 

\noindent in $\mathcal{B}^*$ of $s_\infty \lambda_{c_\infty\Phi}$, $s_\infty(\Phi\ast_\infty i_{t_\infty\Phi})$, and $s_\infty \lambda_{d_\infty\Phi}^{-1}$. Now, since $\lambda_{c_\infty\Phi}$ and $\lambda_{d_\infty\Phi}$ are globular, the composition

\[s_\infty \lambda_{c_\infty\Phi}s_\infty(\Phi\ast_\infty i_{t_\infty\Phi})s_\infty \lambda_{d_\infty\Phi}^{-1}\]

\noindent is equal to the source $s_\infty\Phi\ast_\infty i_{t_\infty\Phi}$ of $\Phi\ast_\infty i_{t_\infty\Phi}$, which is equal to the source $s_\infty\Phi$ of $\Phi$. We conclude that the source functor $s_\infty$ on $F_\infty$ is compatible with relation 3 in the definition of $R_\infty$. An analogous argument proves that the target functor $t_\infty$ is compatible with relation 4 in the definition of $R_\infty$. Further, an analogous argument proves that the source and target functors $s_\infty$ and $t_\infty$ in $F_\infty$ are compatible with relation 5 in the definition of $R_\infty$.
\end{proof}


\

\noindent\textit{The main construction: Observations on lemma \ref{CompatibilityRInfinity}}

\

\noindent Before presenting the definition of the free globularily generated double category we present a few preliminary observations on lemma \ref{CompatibilityRInfinity}.

\

\begin{obs}\label{ObservationFinal1}
Let $\mathcal{B}$ be a decorated bicategory. By lemma 3.9 the functors $s_\infty$ and $t_\infty$ descend to functors from $V_\infty$ to $\mathcal{B}^*$. We keep denoting these functors by $s_\infty$ and $t_\infty$. Moreover, the composition operation function $\ast_\infty$ descends to a function from $H_\infty\times_{\mbox{Hom}_{\mathcal{B}^*}}H_\infty$ to $H_\infty$ such that, by relation 1 in the definition of $R_\infty$, together with the composition operation function for horizontal morphisms in $\mathcal{B}$ forms a functor from $V_\infty\times_{\mathcal{B}^*}V_\infty$ to $V_\infty$. We denote this functor by $\ast_\infty$. 
\end{obs}



\begin{obs}\label{Observationfinal2}
Let $\mathcal{B}$ be a decorated bicategory. Let $k$ be a positive integer. In that case the relation $R_\infty$ restricts to an equivalence relation in $E_k$. We denote by $H_k$ the quotient $E_k/R_\infty$ of $E_k$ modulo $R_\infty$. Moreover, $R_\infty$ restricts to an equivalence relation on the collection of morphisms Hom$_{F_k}$ of $F_k$. The relation $R_\infty$ is compatible with the domain and codomain functions $d_k$ and $c_k$ of $F_k$ and is thus compatible with the category structure of $F_k$. We denote by $V_k$ the quotient $F_k/R_\infty$ of $F_k$ modulo $R_\infty$ and keep denoting by $d_k,c_k$, and $\bullet_k$ the domain, the codomain, and the composition operation functions in $V_k$. The functors $s_{k+1}$ and $t_{k+1}$ are compatible with $R_\infty$ and thus induce functors from $V_k$ to the decoration $\mathcal{B}^*$ of $\mathcal{B}$. We keep denoting these functors by $s_{k+1}$ and $t_{k+1}$ respectively. Finally, observe that the function $\ast_k$ is compatible with $R_\infty$ and thus defines a function from the set of all morphisms of $V_k\times_{\mathcal{B}^*}V_k$ to the set of morphisms of $V_k$. This function, together with the composition operation function for horizontal morphisms in $\mathcal{B}$ forms a functor from $V_k\times_{\mathcal{B}^*}V_k$ to $V_k$. We keep denoting this functor by $\ast_k$.   
\end{obs}

\begin{obs}\label{Observationfinal3}
Let $\mathcal{B}$ be a decorated bicategory. Let $m,k$ be positive integers such that $m$ is less than or equal to $k$. In that case $H_m$ is contained in $H_k$, the category $V_m$ is a subcategory of the category $V_k$, the functors $s_{m+1}$ and $t_{m+1}$ are restrictions to $V_m$ of functors $s_{k+1}$ and $t_{k+1}$, and the functor $\ast_m$ is the restriction to $V_m\times_{\mathcal{B}^*}V_m$ of the functor $\ast_k$. Moreover, the category $V_\infty$ is equal to the limit $\varinjlim V_k$ of the sequence $V_k$, the collection of morphisms $H_\infty$ of $V_\infty$ is equal to the union $\bigcup_{k=1}^\infty H_k$ of the sequence $H_k$, and the functors $s_\infty,t_\infty$ and $\ast_\infty$ are the limits of the sequences of functors $s_k,t_k$ and $\ast_k$ respectively.
\end{obs}

\begin{notation}\label{Preliminarystructurenotation}
Let $\mathcal{B}$ be a decorated bicategory. The pair formed by the function associating the horizontal identity $i_a$ to every object $a$ of $\mathcal{B}$ and the function associating the formal horizontal identity $i_\alpha$ to every vertical morphism $\alpha$ in $\mathcal{B}$ defines a functor from the decoration $\mathcal{B}^*$ of $\mathcal{B}$ to the category $V_\infty$ associated to $\mathcal{B}$. We denote this functor by $i_\infty$. For every positive integer $k$ we denote the codomain restriction to the category $V_k$ of the functor $i_\infty$ by $i_k$. Thus defined, $i_k$ is a functor from the decoration $\mathcal{B}^*$ of $\mathcal{B}$ to the category $V_k$ associated to $\mathcal{B}$ for every positive integer $k$.  
\end{notation}

\begin{lemma}\label{Preliminarylemma}
Let $\mathcal{B}$ be a decorated bicategory. Let $k$ be a positive integer. The functors $s_{k+1},t_{k+1},i_k$, and $\ast_k$ satisfy the following two conditions:

\begin{enumerate}
\item The following two triangles commute
\begin{center}

\begin{tikzpicture}
\matrix(m)[matrix of math nodes, row sep=4em, column sep=4em,text height=1.5ex, text depth=0.25ex]
{\mathcal{B}^*&V_k&\mathcal{B}^*&V_k\\
&\mathcal{B}^*&&\mathcal{B}^*\\};
\path[->,font=\scriptsize]
(m-1-1) edge node[auto] {$i_k$} (m-1-2)
        edge node[left] {$id_{\mathcal{B}^*}$} (m-2-2)			
(m-1-2) edge node[auto] {$s_{k+1}$} (m-2-2)
(m-1-3) edge node[auto] {$i_k$} (m-1-4)
(m-1-4) edge node[auto]	{$t_{k+1}$} (m-2-4)			
(m-1-3) edge node[left] {$id_{\mathcal{B}^*}$} (m-2-4);
\end{tikzpicture}

\end{center}

\item The following two squares commute, where $q^k_1,q^k_2$ denote the left and right projection functors from $V_k\times_{\mathcal{B}^*}V_k$ to $V_k$ respectively.

\begin{center}

\begin{tikzpicture}
\matrix(m)[matrix of math nodes, row sep=4em, column sep=4em,text height=1.5ex, text depth=0.25ex]
{V_k\times_{\mathcal{B}^*}V_k&V_k&V_k\times_{\mathcal{B}^*}V_k&V_k\\
V_k&\mathcal{B}^*&V_k&\mathcal{B}^*\\};
\path[->,font=\scriptsize]
(m-1-1) edge node[auto] {$\ast_k$}  (m-1-2)
        edge node[left] {$q^k_1$}   (m-2-1)			
(m-1-2) edge node[auto] {$s_{k+1}$} (m-2-2)
(m-2-1) edge node[below]{$s_{k+1}$} (m-2-2)
(m-1-3) edge node[auto] {$\ast_k$}  (m-1-4)
(m-1-4) edge node[auto]	{$t_{k+1}$} (m-2-4)			
(m-1-3) edge node[left] {$q^k_2$}   (m-2-3)
(m-2-3) edge node[below]{$s_{k+1}$} (m-2-4);
\end{tikzpicture}

\end{center}

\end{enumerate}
\end{lemma}

\begin{proof}
Let $\mathcal{B}$ be a decorated bicategory. Let $k$ be a positive integer. We wish to prove that conditions 1 and 2 above are satisfied.

We begin by proving that the functors $s_{k+1},t_{k+1}$, and $i_k$ satisfy condition 1 above. Let $\alpha$ be a vertical morphism in $\mathcal{B}$. The formal horizontal identity $i_\alpha$ associated to $\alpha$, that is, the image $i_k\alpha$ of $\alpha$ under the functor $i_k$, is a morphism in $V_1$. From this and from the fact that the sequences $s_{k+1}$ and $t_{k+1}$ satisfy the conditions of proposition 3.2 it follows that the commutativity of triangles in 1 is equivalent to the commutativity of the following triangles:

\begin{center}

\begin{tikzpicture}
\matrix(m)[matrix of math nodes, row sep=4.6em, column sep=5.1em,text height=1.5ex, text depth=0.25ex]
{\mathcal{B}^*&V_1&\mathcal{B}^*&V_1\\
&\mathcal{B}^*&&\mathcal{B}^*\\};
\path[->,font=\scriptsize]
(m-1-1) edge node[auto] {$i_1$} (m-1-2)
        edge node[left] {$id_{V_1}$} (m-2-2)			
(m-1-2) edge node[auto] {$s_{2}$} (m-2-2)
(m-1-3) edge node[auto] {$i_1$} (m-1-4)
(m-1-4) edge node[auto]	{$t_{2}$} (m-2-4)			
(m-1-3) edge node[left] {$id_{V_1}$} (m-2-4);
\end{tikzpicture}

\end{center}

\noindent which follows directly from the definition of the functions $d_1$ and $c_1$. We now prove that the functors $s_{k+1},t_{k+1}$, and $\ast_k$ satisfy condition 2 above. The commutativity of squares in 2 when evaluated on morphisms of $H_k\times_{\mbox{Hom}_{\mathcal{B}^*}}H_k$ follows from observation 3.4. The general commutativity of the squares in condition 2 follows from this and from the fact that all adges involved are functors. This concludes the proof of the lemma.
\end{proof}

\noindent The following corollary follows directly from the previous lemma by taking limits.

\begin{corollary}\label{Preliminarycorollary}
Let $\mathcal{B}$ be a decorated bicategory. The functors $s_\infty,t_\infty, i$, and $\ast_\infty$ satisfy the following two conditions:

\begin{enumerate}

\item The following two triangles commute:

\begin{center}

\begin{tikzpicture}
\matrix(m)[matrix of math nodes, row sep=4em, column sep=4em,text height=1.5ex, text depth=0.25ex]
{\mathcal{B}^*&V_\infty&\mathcal{B}^*&V_\infty\\
&\mathcal{B}^*&&\mathcal{B}^*\\};
\path[->,font=\scriptsize]
(m-1-1) edge node[auto] {$i$} (m-1-2)
        edge node[left] {$id_{\mathcal{B}^*}$} (m-2-2)			
(m-1-2) edge node[auto] {$s_\infty$} (m-2-2)
(m-1-3) edge node[auto] {$i$} (m-1-4)
(m-1-4) edge node[auto]	{$t_\infty$} (m-2-4)			
(m-1-3) edge node[left] {$id_{\mathcal{B}^*}$} (m-2-4);
\end{tikzpicture}

\end{center}

\item The following two squares commute, where $q^\infty_1,q^\infty_2$ denote the left and right projection functors from $V_\infty\times_{\mathcal{B}^*}V_\infty$ to $V_\infty$ respectively.

\begin{center}

\begin{tikzpicture}
\matrix(m)[matrix of math nodes, row sep=4em, column sep=4em,text height=1.5ex, text depth=0.25ex]
{V_\infty\times_{\mathcal{B}^*}V_\infty&V_\infty&V_\infty\times_{\mathcal{B}^*}V_\infty&V_\infty\\
V_\infty&\mathcal{B}^*&V_\infty&\mathcal{B}^*\\};
\path[->,font=\scriptsize]
(m-1-1) edge node[auto] {$\ast_\infty$}  (m-1-2)
        edge node[left] {$q^\infty_1$}   (m-2-1)			
(m-1-2) edge node[auto] {$s_\infty$} (m-2-2)
(m-2-1) edge node[below]{$s_\infty$} (m-2-2)
(m-1-3) edge node[auto] {$\ast_\infty$}  (m-1-4)
(m-1-4) edge node[auto]	{$t_\infty$} (m-2-4)			
(m-1-3) edge node[left] {$q^\infty_2$}   (m-2-3)
(m-2-3) edge node[below]{$s_{k+1}$} (m-2-4);
\end{tikzpicture}

\end{center}

\end{enumerate}
\end{corollary}

\

\noindent \textit{The main construction: The definition}

\

\noindent Let $\mathcal{B}$ be a decorated bicategory. We will denote by $Q_\mathcal{B}$ the pair formed by the decoration $\mathcal{B}^*$ of $\mathcal{B}$ and the category $V_\infty$ associated to $\mathcal{B}$. The following theorem says that we can endow the pair $Q_\mathcal{B}$ with the structure of a globularily generated double category.

\begin{thm}\label{MainTheorem} Let $\mathcal{B}$ be a decorated bicategory. The pair $Q_\mathcal{B}$ together with functors $s_\infty,t_\infty,i$, the functor $\ast_\infty$, and the collection of left and right identity transformations, and associator of $\mathcal{B}$, is a double category. Moreover, with this structure, the double cateogry $Q_\mathcal{B}$, is globularily generated.
\end{thm}

\begin{proof}
Let $\mathcal{B}$ be a decorated bicategory. We wish to prove in this case that the pair $Q_\mathcal{B}$, together with the functors $s_\infty,t_\infty,i$, the functor $\ast_\infty$, and the collection of left and right identity transformations and associator of $\mathcal{B}$ is a globularily generated double category.

The formal horizontal identity functor $i$ and the horizontal composition functor $\ast_\infty$ are compatible with the functors $s_\infty$ and $t_\infty$ by corollary 3.15. The collections of left identity transformations and right identity transformations of $\mathcal{B}$ form a natural transformation from $\ast_\infty(is_\infty\times id_{V_\infty})$ to the identity endofunctor $id_{V_\infty}$ of $V_\infty$ and a natural transformation from $\ast_\infty(id_{V_\infty}\times it_\infty)$ to the identity endofunctor $id_{V_\infty}$ of $V_\infty$ respectively by the fact that morphisms in $V_\infty$ satisfy relation 4 in the definition of $R_\infty$. The collection of associators of $\mathcal{B}$ forms a natural transformation from the composition $\ast_\infty(\ast_\infty\times id_{V_\infty})$ to the composition $\ast_\infty(id_{V_\infty}\times \ast_\infty)$ by the fact that morphisms in $V_\infty$ satisfy relation 5 in the definition of $R_\infty$. The left and right identity and the associator relations for $Q_\mathcal{B}$ again follow from the fact that morphisms in $V_\infty$ satisfy relations 4 and 5 in the definition of $R_\infty$. The fact that the pair formed by the functor $i$ and functor $\ast$ satisfies Mc Lane's triangle and pentagon relations with respect to the left and right identity transformations and associator follows from the fact that the components of the left and right identity transformations and associator satisfy mcLane's axioms for the bicategory $\mathcal{B}$. This proves that $Q_\mathcal{B}$ with the structure described is a double category. A straightforward induction argument proves that for every positive integer $k$ every morphism of $V_k$ is globularily generated in $Q_\mathcal{B}$, from which it follows that the double category $Q_\mathcal{B}$ is globularily generated. This concludes the proof of the theorem.  
\end{proof}

\begin{definition}\label{MainDefinition}
Let $\mathcal{B}$ be a decorated bicategory. We call the globularily generated double category $Q_\mathcal{B}$ the free globularily generated double category associated to $\mathcal{B}$.
\end{definition}

\noindent Lemma \ref{Preliminarylemma} provides the free globularily generated double category $Q_\mathcal{B}$ associated to a decorated bicategory $\mathcal{B}$ with a filtration $\left\{V_k\right\}$ of its category of morphisms $Q_{\mathcal{B}_1}$. We call this filtration the \textbf{free vertical filtration} associated to $Q_\mathcal{B}$. We use this filtration to define numerical invariants for $Q_\mathcal{B}$. Given a square $\varphi$ in $Q_\mathcal{B}$ we say that $\varphi$ is of \textbf{free length} $k$, $\ell_{free}\varphi=k$ in symbols, if $\varphi$ is a morphism in $V_k$ and $\varphi$ is not a morphism in $V_{k-1}$. Further, we say that $Q_\mathcal{B}$ has free vertical lenght $k\in \mathbb{N}\cup\left\{\infty\right\}$, $\ell_{free} Q_\mathcal{B}$ in symbols, if $k$ is the supremum of all free vertical lengths of squares in $Q_\mathcal{B}$. The free vertical filtration $\left\{V_k\right\}$ of $Q_\mathcal{B}$ might differ from the vertical filtration $\left\{V^{Q_\mathcal{B}}_k\right\}$ associated to $Q_\mathcal{B}$ as a globularily generated double category in \cite{yo1}. The free length $\ell_{free}\varphi$ of a square $\varphi$ in $Q_\mathcal{B}$ might differ from the length $\ell\varphi$ of $\varphi$ and correspondingly the free length $\ell_{free}Q_\mathcal{B}$ of $Q_\mathcal{B}$ might differ from the length $\ell Q_\mathcal{B}$ of $Q_\mathcal{B}$ as a globularily generated double category. In section \ref{s3} and section \ref{s4} we study situations in which the free vertical filtration and the usual filtration of a free globularily generated double category coincide.

Using arguments analogous as those employed in the proof of \cite[lemma 4.2]{yo1} it is easily proven that every free length 1 square $\varphi$ in the free globularily generated double category associated to a decorated bicategory $\mathcal{B}$ admitting a pictorial representation as:

\begin{center}

\begin{tikzpicture}
\matrix(m)[matrix of math nodes, row sep=4em, column sep=4em,text height=1.5ex, text depth=0.25ex]
{a&a\\
b&b\\};
\path[->,font=\scriptsize]
(m-1-1) edge node [above]{$\alpha$} (m-1-2)
        edge node [left]{$f$} (m-2-1)
(m-2-1) edge node[below]{$\beta$} (m-2-2)			
(m-1-2) edge node[right]{$f$} (m-2-2)
(m-1-1) edge [white] node[black][fill=white]{$\varphi$} (m-2-2);
\end{tikzpicture}
\end{center}

\noindent admits a factorization as a vertical composition of the form:

\[\psi_{k}\bullet_{\infty}i_{f_k}\bullet_{\infty}\psi_{k-1}\dots\psi_1\bullet_\infty i_{f_1}\bullet_\infty\psi_0\]

\noindent where $f_i:a_{i-1}\to a_i$ is a morphisms in $\mathcal{B}^*$ for every $1\leq i\leq k$, $\psi_i$ is a globular square, in $\mathcal{B}$ of the form:

\begin{center}

\begin{tikzpicture}
\matrix(m)[matrix of math nodes, row sep=4em, column sep=4em,text height=1.5ex, text depth=0.25ex]
{a_{i-1}&a_{i-1}\\
a_i&a_i\\};
\path[->,font=\scriptsize]
(m-1-1) edge [red] node {} (m-1-2)
        edge[blue] node {} (m-2-1)
(m-2-1) edge [red]node{} (m-2-2)			
(m-1-2) edge [blue] node {} (m-2-2)
(m-1-1) edge [white] node[black][fill=white]{$\psi_i$} (m-2-2);
\end{tikzpicture}

\end{center}

\noindent for every $1\leq i\leq k-1$, and where $\psi_0,\psi_k$ are globular squares of the form:

\begin{center}

\begin{tikzpicture}
\matrix(m)[matrix of math nodes, row sep=4em, column sep=4em,text height=1.5ex, text depth=0.25ex]
{a&a&a&a\\
a&a&a&a\\};
\path[->,font=\scriptsize]
(m-1-1) edge node [above]{$\alpha$} (m-1-2)
        edge[blue] node {} (m-2-1)
(m-2-1) edge [red]node{} (m-2-2)			
(m-1-2) edge [blue] node {} (m-2-2)
(m-1-1) edge [white] node[black][fill=white]{$\psi_0$} (m-2-2)

(m-1-3) edge[red] node {} (m-1-4)
        edge[blue] node {} (m-2-3)
(m-2-3) edge node[below]{$\beta$} (m-2-4)			
(m-1-4) edge [blue] node {} (m-2-4)
(m-1-3) edge [white] node[black][fill=white]{$\psi_1$} (m-2-4);
\end{tikzpicture}

\end{center}

\noindent We will make strong use of this fact in the rest of the paper.

Let $k$ be a field. We will understand for a $k$-linear decorated bicategory a decorated bicategory $\mathcal{B}$ such that both the underlying bicategory and the decoration of $\mathcal{B}$ are endowed with $k$-linear structures, where we understand for a $k$-linear structure on a bicategory $\mathcal{B}$ a structure of $k$-vector space for the set of 2-cells of the form:

\begin{center}

\begin{tikzpicture}
\matrix(m)[matrix of math nodes, row sep=4em, column sep=4em,text height=1.5ex, text depth=0.25ex]
{a&b\\};
\path[->,font=\scriptsize]
(m-1-1) edge [bend left=45] node [above]{$\alpha$}(m-1-2)
        edge [bend right=45] node[below]{$\beta$}(m-1-2)
        edge [white]node[black][fill=white]{$\varphi$}(m-1-2);
\end{tikzpicture}
\end{center}

\noindent for any pair of 1-cells $\alpha,\beta$ in $\mathcal{B}$ fitting in a diagram as above, in such a way that all the corresponding structure and coherence data is $k$-linear. When the decorated bicategory $\mathcal{B}$ is endowed with a linear structure, the free globularily double category construction can be modified, in the obvious way, such that the resulting double category, which we denote $Q^k_\mathcal{B}$, is endowed with the structure of a category internal to $k$-linear categories. We study this modification of the globularily generated double category construction in the context of categorical aspects of the representation theory of von Neumann algebras section \ref{s6}.

Finally, it is natural to expect relations between the free globularily generated double category construction and the free double category construction of Dawson and Par\'e \cite{DawsonPareFree}. Let $G$ be a reflexive double graph. Write $\mathcal{B}_G$ for the decorated horizontalization $H^*F(D)$ of the free double category generated by $D$. From the way the globlarily generated double category was constructed it is easily seen that the free double category $F(D)$ generated by $D$ and $Q_{\mathcal{B}_D}$ are related through the equation $Q_{\mathcal{B}_D}=\gamma F(D)$.

\section{Free globularily generated internalizations}\label{s3}

\noindent In this section we study situations in which the free globularily generated double category construction provides solutions to problem \ref{prob}. The following example shows that the free globularily generated double category construction not always provides solutions to this problem. Given a monoid $M$ we write $\Omega M$ for the delooping category of $M$, i.e. $\Omega M$ is the category with a single object $\ast$ whose monoid of endomorphisms $Emd_{\Omega M}(\ast)$ is $M$. Given a monoidal category $D$ we write $2D$ for the delooping bicategory of $D$, i.e. $\Omega D$ is the single object bicategory whose monoidal category of endomorphisms is $D$. Observe that given a monoid $M$ the delooping category $\Omega M$ of $M$ admits the structure of strict monoidal category if and only if $M$ is commutative by the Eckman-Hilton argument \cite{EckmanHilton}.

\begin{ex}\label{groupdecorationsnotalwaysinternalizations}
Let $\mathcal{B}$ be the decorated bicategory whose underlying bicategory is the single 0-cell and single 1-cell 2-category $2\Omega\mathbb{Z}_2$ and whose decoration $\mathcal{B}^*$ is the delooping category $\Omega\mathbb{Z}_2$. The free globularily generated double category $Q_\mathcal{B}$ associated to $\mathcal{B}$ does not provide solutions to problem \ref{prob} for $\mathcal{B}$. To see this consider the square:

\begin{center}

\begin{tikzpicture}
\matrix(m)[matrix of math nodes, row sep=4em, column sep=4em,text height=1.5ex, text depth=0.25ex]
{\ast&\ast\\
\ast&\ast\\
\ast&\ast\\
\ast&\ast\\};
\path[->,font=\scriptsize]
(m-1-1) edge [red] node {} (m-1-2)
        edge node [left]{$-1$} (m-2-1)
(m-1-2) edge node[right]{$-1$} (m-2-2)			
(m-2-1) edge [red] node {} (m-2-2)
(m-1-1) edge [white] node [black][fill=white]{$i_{-1}$}(m-2-2)

(m-2-1) edge [blue] node{} (m-3-1)
(m-2-2) edge [blue]node{} (m-3-2)			
(m-3-1) edge [red]node {} (m-3-2)
(m-2-1) edge [white]node[black][fill=white]{$-1$}(m-3-2)

(m-3-1) edge node[left]{$-1$} (m-4-1)
(m-3-2) edge node[right]{$-1$} (m-4-2)			
(m-4-1) edge [red]node {} (m-4-2)
(m-3-1) edge [white]node[black][fill=white]{$i_{-1}$}(m-4-2);
\end{tikzpicture}

\end{center}

\noindent in $Q_\mathcal{B}$, where we write $\mathbb{Z}_2$ multiplicatively, i.e. $\mathbb{Z}_2=\left\{\pm 1\right\}$. We denote this square by $\varphi$. Thus defined $\varphi$ satisfies the equations $t_\infty\varphi=s_\infty\varphi=(-1)^2=1$ and is thus globular in $Q_\mathcal{B}$. The only globular squares of $\mathcal{B}$ are the squares:

\begin{center}

\begin{tikzpicture}
\matrix(m)[matrix of math nodes, row sep=4em, column sep=4em,text height=1.5ex, text depth=0.25ex]
{\ast&\ast&\ast&\ast\\
\ast&\ast&\ast&\ast\\};
\path[->,font=\scriptsize]
(m-1-1) edge [red] node {} (m-1-2)
        edge [blue]node{} (m-2-1)
(m-1-2) edge [blue]node{} (m-2-2)			
(m-2-1) edge [red] node {} (m-2-2)
(m-1-1) edge [white] node [black][fill=white]{$i_1$}(m-2-2)

(m-1-3) edge [red] node{} (m-1-4)
(m-1-3) edge [blue]node{} (m-2-3)			
(m-1-4) edge [blue]node {} (m-2-4)
(m-2-3) edge [red]node {}(m-2-4)
(m-1-3) edge [white]node[black][fill=white]{$-1$}(m-2-4);
\end{tikzpicture}

\end{center}

\noindent none of which is equal to $\varphi$. We conclude that $\varphi$ is a globular square in $H^*Q_\mathcal{B}$ not contained in $\mathcal{B}$ and thus that $Q_\mathcal{B}$ does not provide a solution to problem \ref{prob} for $\mathcal{B}$.
\end{ex}

\noindent We now proivde conditions under which the free globularily generated double category $Q_\mathcal{B}$ associated to a decorated bicategory $\mathcal{B}$ does provide a solution to problem \ref{prob}. We say that a category $\mathcal{B}^*$ is reduced when the underlying groupoid of $\mathcal{B}^*$ is discrete, i.e. we say that $\mathcal{B}^*$ is reduced when the only isomorphisms of $\mathcal{B}$ are identities. Examples of reduced categories are delooping categories $\Omega M$ where $M$ is a reduced monoid, e.g. $M$ is any submonoid of $\mathbb{N}$, categories associated to partially ordered sets, e.g. Open$(X)$ for a topological space $X$, and path categories associated to graphs. The following proposition says that the free globularily generated double category associated to a decorated bicategory with reduced decoration provides solutions to problem \ref{prob}.

\begin{prop}\label{reducedprop}
Let $\mathcal{B}$ be a decorated bicategory. If $\mathcal{B}^*$ is reduced then the equation $H^*Q_\mathcal{B}=\mathcal{B}$ holds.
\end{prop}

\begin{proof}
Let $\mathcal{B}$ be a decorated bicategory. Assume that $\mathcal{B}^*$ is reduced. We wish to prove that in this case the equation $H^*Q_\mathcal{B}=\mathcal{B}$ holds.

We proceed by induction on $k$ to prove that every globular square in $V_k$ is a globular square in $\mathcal{B}$. We begin by proving the statement for $k=1$. Let $\varphi$ be a globular square in $V_1$. By \cite{yo1} if $\varphi$ is not a horizontal endomorphism of $Q_\mathcal{B}$ then $\varphi$ is a globular square in $\mathcal{B}$. We thus assume that $\varphi$ is a horizontal endomorphism in $Q_\mathcal{B}$. Represent $\varphi$ pictorially as:

\begin{center}

\begin{tikzpicture}
\matrix(m)[matrix of math nodes, row sep=4em, column sep=4em,text height=1.5ex, text depth=0.25ex]
{a&a\\
a&a\\};
\path[->,font=\scriptsize]
(m-1-1) edge node [above]{$\alpha$} (m-1-2)
        edge[blue] node {} (m-2-1)
(m-2-1) edge node[below]{$\beta$} (m-2-2)			
(m-1-2) edge [blue] node {} (m-2-2)
(m-1-1) edge [white] node[black][fill=white]{$\varphi$} (m-2-2);
\end{tikzpicture}

\end{center}

\noindent In that case $\varphi$ can be written as a vertical composition, in $Q_\mathcal{B}$ of the form:

\[\psi_{k}\bullet_{\infty}i_{f_k}\bullet_{\infty}\psi_{k-1}\dots\psi_1\bullet_\infty i_{f_1}\bullet_\infty\psi_0\]

\noindent where $f_i:a_{i-1}\to a_i$ is a morphisms in $\mathcal{B}^*$ for every $1\leq i\leq k$, $\psi_i$ is a globular square, in $\mathcal{B}$ of the form:

\begin{center}

\begin{tikzpicture}
\matrix(m)[matrix of math nodes, row sep=4em, column sep=4em,text height=1.5ex, text depth=0.25ex]
{a_{i-1}&a_{i-1}\\
a_i&a_i\\};
\path[->,font=\scriptsize]
(m-1-1) edge [red] node {} (m-1-2)
        edge[blue] node {} (m-2-1)
(m-2-1) edge [red]node{} (m-2-2)			
(m-1-2) edge [blue] node {} (m-2-2)
(m-1-1) edge [white] node[black][fill=white]{$\psi_i$} (m-2-2);
\end{tikzpicture}

\end{center}

\noindent for every $1\leq i\leq k-1$, and where $\psi_0,\psi_k$ are globular squares of the form:

\begin{center}

\begin{tikzpicture}
\matrix(m)[matrix of math nodes, row sep=4em, column sep=4em,text height=1.5ex, text depth=0.25ex]
{a&a&a&a\\
a&a&a&a\\};
\path[->,font=\scriptsize]
(m-1-1) edge node [above]{$\alpha$} (m-1-2)
        edge[blue] node {} (m-2-1)
(m-2-1) edge [red]node{} (m-2-2)			
(m-1-2) edge [blue] node {} (m-2-2)
(m-1-1) edge [white] node[black][fill=white]{$\psi_0$} (m-2-2)

(m-1-3) edge[red] node {} (m-1-4)
        edge[blue] node {} (m-2-3)
(m-2-3) edge node[below]{$\beta$} (m-2-4)			
(m-1-4) edge [blue] node {} (m-2-4)
(m-1-3) edge [white] node[black][fill=white]{$\psi_1$} (m-2-4);
\end{tikzpicture}

\end{center}

\noindent From the fact that $\varphi$ is globular it follows that the composition $f_k\dots f_1$ is equal to $id_a$. By the fact that $\mathcal{B}^*$ is reduced it follows that $a_i=a$ and $f_i=id_a$ for every $1\leq i\leq k$. We conclude that $\varphi$ is a vertical composition, in $Q_\mathcal{B}$ of globular squares of $\mathcal{B}$ and thus is a globular square in $\mathcal{B}$.

Let $k>1$. Suppose that the result is true for all positive integers $m$ such that $m<k$, i.e. suppose that every globular square in $V_m$ is a globular square in $\mathcal{B}$ for every $m<k$. Let $\varphi$ be a globular square in $V_k$. We prove that $\varphi$ is a globular square in $\mathcal{B}$. Assume first that $\varphi\in H_k$. In that case $\varphi$ admits a decomposition as $\psi_n\ast_\infty\dots\ast_\infty\psi_1$ where $\psi_i$ is a square in $V_{k-1}$ for every $i$. The horizontal composition of non-globular squares is never globular, thus in the above case $\psi_i$ is globular for every $i$. Bu the induction hypothesis $\varphi$ is in this case horizontal composition of globular squares in $\mathcal{B}$ and is thus a globular square in $\mathcal{B}$. Now suppose that $\varphi$ is a general square of $V_k$. In that case $\varphi$ admits a decomposition as vertical composition $\psi_n\bullet_\infty\dots\bullet_\infty\psi_1$ where $\varphi_i$ is a globular square in $H_k$ for every $i$. By the above argument every $\psi_i$ is a globular square in $\mathcal{B}$ and thus $\varphi$ is a globular square in $\mathcal{B}$. This concludes the proof of the proposition. 

\end{proof}

\noindent In the cases in which the free globularily generated double category $Q_\mathcal{B}$ associated to a decorated bicategory $\mathcal{B}$ is not an internalization of $\mathcal{B}$ we can always associate to $\mathcal{B}$ a larger decorated bicategory for which the free globularily generated double category construction does provide solutions to problem \ref{prob}. To see we first prove the following proposition.

\begin{prop}\label{freeidempotent}
Let $\mathcal{B}$ be a decorated bicategory. In that case the equation $Q_{H^*Q_\mathcal{B}}=Q_\mathcal{B}$ holds. 
\end{prop}

\begin{proof}
Let $\mathcal{B}$ be a decorated bicategory. We wish to prove that the equation $Q_{H^*Q_\mathcal{B}}=Q_\mathcal{B}$ holds.

The categories of objects of $Q_\mathcal{B}$ and $Q_{H^*Q_\mathcal{B}}$ are both equal to $\mathcal{B}^*$. The collections of horizontal morphisms of $Q_\mathcal{B}$ and $Q_{H^*Q_\mathcal{B}}$ are both equal to $\mathcal{B}_1$. The collection of squares of $Q_{\mathcal{B}}$ is clearly contained in $Q_{H^*Q_\mathcal{B}}$. To prove the proposition we thus need to prove that every square of $Q_{H^*Q_\mathcal{B}}$ is a square in $Q_\mathcal{B}$. For every positive integer $k$ we will write $\tilde{V}_k$ and $\tilde{H}_k$ for the category $V_k$ associated to $H^*Q_\mathcal{B}$ and for the set $H_k$ associated to $H^*Q_\mathcal{B}$ in lemma \ref{Preliminarylemma}. We prove, by induction on $k$, that every square in $\tilde{V}_k$ is a square in $Q_\mathcal{B}$. 

Let $\varphi$ be a square in $\tilde{V}_1$. In that case $\varphi$ admits a decomposition as:

\[\varphi=\psi_n \bullet_\infty i_{f_n}\bullet_\infty\psi_{n-1}\dots\psi_1\bullet_\infty i_{f_1}\bullet_\infty\psi_0\]

\noindent where $\psi_0,...,\psi_n$ and $f_1,...,f_n$ are as in the proof of proposition \ref{reducedprop}. Observe that each $i_{f_j}$ is a square in $V_1$ and each $\psi_j$ is a square in some $V_{k_j}$ and thus is a square of $Q_\mathcal{B}$ for every $i$. $\varphi$ is thus a square in $Q_\mathcal{B}$. 

Let $k$ be a positive integer such that $k>1$. Suppose that the result is true for every $m\leq k$. We prove that every square in $\tilde{V}_k$ is a square in $Q_\mathcal{B}$. Let $\varphi$ be a square in $\tilde{V}_k$. Suppose first that $\varphi\in \tilde{H}_k$. In that case $\varphi$ admits a decomposition as $\varphi=\psi_n\ast_\infty\dots\ast_\infty\psi_1$ where $\psi_1,...,\psi_n$ are squares in $\tilde{V}_{k-1}$ and thus are squares in $Q_\mathcal{B}$. The square $\varphi$ is thus a square in $Q_\mathcal{B}$. Suppose now that $\varphi$ is a general square in $\tilde{V}_k$. In that case $\varphi$ admits a decomposition as $\varphi=\psi_n\bullet_\infty\dots\bullet_\infty\psi_1$ where $\psi_i$ is a square in $\tilde{H}_k$ and is thus a square in $Q_\mathcal{B}$. The square $\varphi$ is thus a square in $Q_\mathcal{B}$. This concludes the proof of the proposition.
\end{proof}

\noindent Proposition \ref{freeidempotent} says that the operation of 'taking the free globularily generated double category' is idempotent, i.e. stops at order 2. We have the following immediate corollary.

\begin{corollary}\label{saturationcorollary}
Let $\mathcal{B}$ be a decorated bicategory. In that case $Q_\mathcal{B}$ is an internalization of $H^*Q_\mathcal{B}$.
\end{corollary}

\noindent Given a decorated bicategory $\mathcal{B}$ we call the decorated bicategory $H^*Q_\mathcal{B}$ the \textbf{saturation} of $\mathcal{B}$. We say that a decorated bicategory $\mathcal{B}$ is \textbf{saturated} whenever $\mathcal{B}$ is equal to its saturation $H^*Q_\mathcal{B}$. While the free globularily generated double category $Q_\mathcal{B}$ might not always provide a solution to problem \ref{prob} for the decorated bicategory $\mathcal{B}$ provided as set of initial conditions, the free globularily generated double category $Q_{H^*Q_\mathcal{B}}$ always provides a solution to problem \ref{prob} for the saturation $H^*Q_\mathcal{B}$ of $\mathcal{B}$. We compute saturations of certain decorated bicategories in the following sections. Observe that if a decorated bicategory $\mathcal{B}$ is saturated then the vertical filtration and the free vertical filtration of $Q_\mathcal{B}$ coincide and thus the free vertical lenght and the usual vertical length of squares in $Q_\mathcal{B}$ and of $Q_\mathcal{B}$ itself coincide. Decorated bicategories with reduced decorations are saturated by proposition \ref{reducedprop}.

\section{Length}\label{s4}

\noindent In this section we apply the free globularily generated double category construction to provide examples of double categories of non-trivial length. All the examples of double categories considered in \cite{yo1}, i.e. trivial double categories, and double categories of bordisms, algebras and von Neumann algebras are proven to be of length 1. The following example proves that the concept of length of a double category is non-trivial by explicitly constructing a double category of length equal to 2.

\begin{ex} \label{exlength2}
Let $\mathcal{B}$ denote the following 2-category: $\mathcal{B}$ has three objects $a,b,c$ and only horizontal identity 1-cells. All 2-cells in $\mathcal{B}$ will be identities except for one vertical endomorphism cell of $i_b$. This cell, together with the identity 2-cell of $i_b$ will form the group $\mathbb{Z}_2$ under both horizontal and vertical composition. Pictorially $\mathcal{B}$ is represented by the diagram:

\begin{center}

\begin{tikzpicture}
\matrix(m)[matrix of math nodes, row sep=4em, column sep=4em,text height=1.5ex, text depth=0.25ex]
{a&a\\
b&b\\
c&c\\};
\path[->,font=\scriptsize]
(m-1-1) edge [red,bend left=45] node {}(m-1-2)
        edge [red,bend right=45] node{}(m-1-2)
        edge [white]node[black][fill=white]{$id$}(m-1-2)

(m-2-1) edge [red,bend left=45] node {}(m-2-2)
        edge [red,bend right=45] node{}(m-2-2)
        edge [white]node[black][fill=white]{$\mathbb{Z}_2$}(m-2-2)

(m-3-1) edge [red,bend left=45] node {}(m-3-2)
        edge [red,bend right=45] node{}(m-3-2)
        edge [white]node[black][fill=white]{$id$}(m-3-2) ;
\end{tikzpicture}
\end{center}

\noindent Now decorate $\mathcal{B}$ with the following category $\mathcal{B}^*$: $\mathcal{B}^*$ has non-identity morphism $\alpha:a\to b,\beta,\beta':b\to C$ and $\gamma:a\to c$ satisfying the relation $\beta\alpha=\gamma=\beta'\alpha$. We represent $\mathcal{B}^*$ pictorially as:

\begin{center}

\begin{tikzpicture}
\matrix(m)[matrix of math nodes, row sep=4em, column sep=4em,text height=1.5ex, text depth=0.25ex]
{a\\
b\\
c\\};
\path[->,font=\scriptsize]
(m-1-1) edge node[left]{$\alpha$} (m-2-1)
        edge[bend right=65]node [left]{$\gamma$}(m-3-1)

(m-2-1) edge [bend right=45]node[left]{$\beta$} (m-3-1)
        edge [bend left=45]node[right]{$\beta'$}(m-3-1);
\end{tikzpicture}
\end{center}

\noindent We claim that $\ell Q_\mathcal{B}=2$. Observe first that since $\mathcal{B}^*$ is reduced, it is enough to prove that $\ell_{free}Q_\mathcal{B}=2$. We exhibit a pair of horizontally composable squares $\varphi,\psi$ in $Q_\mathcal{B}$ of vertical length 1 such that $\varphi\ast_\infty\psi$ is not a morphism in $V_1$. Write $\varphi$ and $\psi$ for the squares pictorially represented as:

\begin{center}

\begin{tikzpicture}
\matrix(m)[matrix of math nodes, row sep=4em, column sep=4em,text height=1.5ex, text depth=0.25ex]
{a&a&a&a\\
b&b&b&b\\
b&b&b&b\\
c&c&c&c\\};
\path[->,font=\scriptsize]
(m-1-1) edge [red] node {} (m-1-2)
        edge node [left]{$\alpha$} (m-2-1)
(m-1-2) edge node[right]{$\alpha$} (m-2-2)			
(m-2-1) edge [red] node {} (m-2-2)
(m-1-1) edge [white] node [black][fill=white]{$i_\alpha$}(m-2-2)

(m-2-1) edge [blue] node{} (m-3-1)
(m-2-2) edge [blue]node{} (m-3-2)			
(m-3-1) edge [red]node {} (m-3-2)
(m-2-1) edge [white]node[black][fill=white]{$-1$}(m-3-2)

(m-3-1) edge node[left]{$\beta$} (m-4-1)
(m-3-2) edge node[right]{$\beta$} (m-4-2)			
(m-4-1) edge [red]node {} (m-4-2)
(m-3-1) edge [white]node[black][fill=white]{$i_\beta$}(m-4-2)

(m-1-3) edge [red] node {} (m-1-4)
        edge node [left]{$\alpha$} (m-2-3)
(m-1-4) edge node[right]{$\alpha$} (m-2-4)			
(m-2-3) edge [red] node {} (m-2-4)
(m-1-3) edge [white] node [black][fill=white]{$i_\alpha$}(m-2-4)

(m-2-3) edge [blue] node{} (m-3-3)
(m-2-4) edge [blue]node{} (m-3-4)			
(m-3-3) edge [red]node {} (m-3-4)
(m-2-3) edge [white]node[black][fill=white]{$-1$}(m-3-4)

(m-3-3) edge node[left]{$\beta'$} (m-4-3)
(m-3-4) edge node[right]{$\beta'$} (m-4-4)			
(m-4-3) edge [red]node {} (m-4-4)
(m-3-3) edge [white]node[black][fill=white]{$i_{\beta'}$}(m-4-4);
\end{tikzpicture}

\end{center}

\noindent Thus defined $\varphi,\psi$ satisfy the equation $t_\infty\varphi=s_\infty\varphi=\gamma$ and thus are horizontally composable in $Q_\mathcal{B}$. We prove that $\varphi\ast_\infty\psi$ is not a morphism in $V_1$. To do this we first observe that $s_\infty\varphi\ast_\infty\psi=t_\infty\varphi\ast_\infty\psi=\gamma$. The only squares in $V_1$ with source and target equal to $\gamma$ are $i_\gamma$ and the squares $\varphi$ and $\psi$. To see that $\varphi\ast_\infty\psi$ is not equal to any of these three squares in $Q_\mathcal{B}$ observe that while $\varphi\ast_\infty\varphi=i_\gamma$ and $\psi\ast_\infty\psi=i_\gamma$, $\varphi\ast_\infty\psi$ satisfies the relations $\varphi\ast_\infty(\varphi\ast_\infty\psi)=\psi$ and $(\varphi\ast_\infty\psi)\ast_\infty\psi=\varphi$. From this and from the obvious fact that $\varphi\ast_\infty\psi$ is not equal to $i_\gamma$ it follows that $\varphi\ast_\infty\psi$ does not have vertical length equal to 1. $\ell Q_\mathcal{B}$ thus satisfies the inequality $\ell Q_\mathcal{B}\geq 2$, but it is obvious from the definition of $\mathcal{B}$ and $\mathcal{B}^*$ that $\ell Q_\mathcal{B}\leq 2$. We conclude that $\ell Q_\mathcal{B}=2$.

\end{ex}

\noindent The above example shows that the concept of vertical length of a double category is not trivial. We explain how to extend the construction presented in example \ref{exlength2} to a sequence of saturated decorated bicategories $\mathcal{B}_k$ such that $\ell Q_\mathcal{B}=k$ for every $k$.

Let $k$ be a positive integer. We make the underlying bicategory of $\mathcal{B}_k$ to be the $k+2$ vertex/2-cell version of the 2-category employed in the construction of example \ref{exlength2}. The underlying 2-category of $\mathcal{B}_k$ is represented by a vertical sequence of $k+2$ diagrams of the form:

\begin{center}

\begin{tikzpicture}
\matrix(m)[matrix of math nodes, row sep=4em, column sep=4em,text height=1.5ex, text depth=0.25ex]
{\bullet&\bullet\\
\bullet&\bullet\\
       &       \\
\bullet&\bullet\\
\bullet&\bullet\\};
\path[->,font=\scriptsize]
(m-1-1) edge [red,bend left=45] node {}(m-1-2)
        edge [red,bend right=45] node{}(m-1-2)
        edge [white]node[black][fill=white]{$id$}(m-1-2)
        
 (m-2-1) edge [red,bend left=45] node {}(m-2-2)
        edge [red,bend right=45] node{}(m-2-2)
        edge [white]node[black][fill=white]{$\mathbb{Z}_2$}(m-2-2)

(m-2-1) edge [white] node [black][fill=white]{$\vdots$}(m-4-2)

(m-4-1) edge [red,bend left=45] node {}(m-4-2)
        edge [red,bend right=45] node{}(m-4-2)
        edge [white]node[black][fill=white]{$\mathbb{Z}_2$}(m-4-2) 
        
(m-5-1) edge [red,bend left=45] node {}(m-5-2)
        edge [red,bend right=45] node{}(m-5-2)
        edge [white]node[black][fill=white]{$id$}(m-5-2)        ;
\end{tikzpicture}
\end{center}

\noindent We define the decoration $\mathcal{B}^*_k$ of $\mathcal{B}_k$. We make $\mathcal{B}^*_k$ to be generated by the graph $G_k$, which we define inductively as follows: We make $G_1$ be the graph generated by the arrows $\alpha,\beta,\beta'$ defining the category $\mathcal{B}^*$ in example \ref{exlength2}. Let $k>1$. Assuming the graph $G_{k-1}$ has been defined, we make the graph $G_k$ to be the graph pictorially represented by the diagram:

\begin{center}

\begin{tikzpicture}
\matrix(m)[matrix of math nodes, row sep=4em, column sep=4em,text height=1.5ex, text depth=0.25ex]
{\bullet\\
\bullet\\
\bullet\\};
\path[->,font=\scriptsize]
(m-1-1) edge node[left]{$\alpha_k$} (m-2-1)

(m-2-1) edge [cyan,bend right=45]node[black][left]{$G_{k-1}$} (m-3-1)
        edge [bend left=45]node[right]{$\beta_k$}(m-3-1);
\end{tikzpicture}
\end{center}

\noindent where the light blue arrow represents the graph $G_{k-1}$. It easily proven that thus defined the graph $G_k$ has $k+2$ verteces and exactly $k+1$ paths of maximal length $k+1$. Let $\mathcal{B}_k^*$ be the category generated by $G_k$ by identifying the maximal paths in each of the $G_m$ for $m\leq k$. Thus defined $G_k$ has a unique maximal path, which we denote by $\gamma_k$. Observe that $\mathcal{B}_1^*$ is the category $\mathcal{B}^*$ of example \ref{exlength2}. Now, assume the existence of a square $\varphi_{k-1}$ in $\mathcal{B}_{k-1}$ of length $k-1$ having $\gamma_{k-1}$ as source and target. Write $\psi_k,\psi'_k$ to be the following two squares of $Q_{\mathcal{B}_k}$:

\begin{center}

\begin{tikzpicture}
\matrix(m)[matrix of math nodes, row sep=4em, column sep=4em,text height=1.5ex, text depth=0.25ex]
{\bullet&\bullet&\bullet&\bullet\\
\bullet&\bullet&\bullet&\bullet\\
\bullet&\bullet&\bullet&\bullet\\
\bullet&\bullet&\bullet&\bullet\\};
\path[->,font=\scriptsize]

(m-1-1) edge [red]node [above]{} (m-1-2)
        edge node [left]{$\alpha_k$} (m-2-1)
(m-2-1) edge[red] node{} (m-2-2)			
(m-1-2) edge node [right]{$\alpha_k$} (m-2-2)
(m-1-1) edge [white] node[black][fill=white]{$i_{\alpha_k}$} (m-2-2)

(m-2-1) edge [blue] node {} (m-3-1)
(m-3-1) edge [red]node[below]{} (m-3-2)			
(m-2-2) edge[blue] node {} (m-3-2)
(m-2-1) edge [white] node[black][fill=white]{$-1$} (m-3-2)

(m-3-1) edge  node [left]{$\gamma_{k-1}$} (m-4-1)
(m-3-2) edge node [right]{$\gamma_{k-1}$} (m-4-2)
(m-4-1) edge[red] node {}(m-4-2)
(m-3-1) edge [white] node[black][fill=white]{$\varphi_{k-1}$} (m-4-2)

(m-1-3) edge [red]node [above]{} (m-1-4)
(m-1-3) edge node [left]{$\alpha_k$}(m-2-3)
(m-2-3) edge [red]node[below]{} (m-2-4)			
(m-1-4) edge node [right]{$\alpha_k$} (m-2-4)
(m-1-3) edge [white] node[black][fill=white]{$i_{\alpha_k}$} (m-2-4)

(m-2-3) edge[blue] node {} (m-3-3)
(m-2-4) edge [blue] node {}(m-3-4)
(m-3-3) edge [red]node {}(m-3-4)
(m-2-3) edge [white] node[black][fill=white]{$-1$} (m-3-4)

(m-3-3) edge node [left]{$\beta_k$} (m-4-3)
(m-3-4) edge node [right]{$\beta_k$} (m-4-4)			
(m-4-3) edge [red] node {} (m-4-4)
(m-3-3) edge [white] node[black][fill=white]{$i_{\beta_k}$} (m-4-4);
\end{tikzpicture}
\end{center}

\noindent Thus defined $\psi_k,\psi'_k$ are of length $k-1$ and by arguments similar to those presented in example \ref{exlength2} the horizontal composition $\psi_k\ast_\infty\psi'_k$ is of length $k$. We write $\varphi_k$ for this square. The free globularily generated double category $Q_{\mathcal{B}_k}$ is thus of length $\geq k$ for every $k$. It is easily seen that $\ell Q_{\mathcal{B}_k}$ is in fact equal to $k$ for every $k$.

Finally, observe that if $\mathcal{B}_\infty$ is the limit $\varinjlim \mathcal{B}_k$, i.e. $\mathcal{B}_\infty$ is equal to the limit of diagram of 2-categories $\mathcal{B}_k$, decorated by the limit of the diagram of categories $\mathcal{B}^*_k$, then $\ell Q_{\mathcal{B}_\infty}=\infty$.

\section{Group decorations}\label{s5}

\noindent In this section we study free globularily generated double categories associated to monoidal categories decorated by groups. We prove that the free globularily generated double category associated to any such decorated bicategory has free lenght equal to 1. Moreover, we prove that in this case the free globularily generated double category construction specializes to the free product operation of groups. We use this to provide explicit descriptions for saturations of such decorated bicategories. We begin by proving the following proposition.

\begin{prop}\label{groupdecorations}
Let $G$ be a group. Let $D$ be a monoidal category. If we write $\mathcal{B}$ for the decorated bicategory $(\Omega G,2D)$ then $\ell_{free}Q_\mathcal{B}=1$.
\end{prop}

\begin{proof}
Let $G$ be a group. Let $D$ be a monoidal category. We wish to prove that the free globularily generated double category $Q_\mathcal{B}$ associated to $\mathcal{B}=(\Omega G,2D)$ is such that $\ell_{free}Q_\mathcal{B}=1$.

We prove that $V_1$ is closed under $\ast_\infty$. Let $\varphi,\psi$ be squares in $V_1$ such that $t\varphi=s\psi$. If $\varphi,\psi$ are globular squares in $\mathcal{B}$ then $\varphi\ast_\infty\psi$ is a globular square in $\mathcal{B}$ and thus is a square in $V_1$. We thus assume that $\phi,\psi$ are not globular squares in $\mathcal{B}$. By results of \cite{yo1} $\varphi,\psi$ are horizontal endomorphisms. Represent $\varphi$ and $\psi$ pictorially as:

\begin{center}

\begin{tikzpicture}
\matrix(m)[matrix of math nodes, row sep=4em, column sep=4em,text height=1.5ex, text depth=0.25ex]
{\ast&\ast&\ast&\ast\\
\ast&\ast&\ast&\ast\\};
\path[->,font=\scriptsize]
(m-1-1) edge node [above]{$a$} (m-1-2)
        edge node [left]{$g$} (m-2-1)
(m-2-1) edge node[below]{$b$} (m-2-2)			
(m-1-2) edge node [right]{$g$} (m-2-2)
(m-1-1) edge [white] node[black][fill=white]{$\varphi$} (m-2-2)

(m-1-3) edge node [above]{$a'$} (m-1-4)
        edge node [left]{$g$} (m-2-3)
(m-2-3) edge node[below]{$b'$} (m-2-4)			
(m-1-4) edge node [right]{$g$} (m-2-4)
(m-1-3) edge [white] node[black][fill=white]{$\psi$} (m-2-4);
\end{tikzpicture}
\end{center}

\noindent where $a,a',b$ and $b'$ are objects in $D$ and $g\in G$. Write $\varphi$ and $\psi$ as vertical compositions of the form

\[\varphi=\varphi_{k+1}\bullet_\infty i_{g_k}\bullet_\infty\dots \bullet_\infty i_{g_1}\bullet_\infty\varphi_0\]

\noindent and

\[\psi=\psi_{s+1}\bullet_\infty i_{g'_s}\bullet_\infty\dots \bullet_\infty i_{g'_1}\bullet_\infty\psi_0\]

\noindent where $g_1,...,g_k,g'_1,...,g'_s$ are elements of $G$ such that $g_1\dots g_k=g=g'_1\dots g'_s$, where $\varphi_1,...,\varphi_k,\psi_1,...,\psi_k\in End_D(1)$, where $\varphi_0,\psi_0$ are morphisms, in $D$, from $a$ to 1 and from $a'$ to 1 respectively, and where $\varphi_{k+1},\psi_{k+1}$ are morphisms, in $D$, from 1 to $b$ and $b'$ respectively. We refer to these decompositions as equations 1 and 2. We make the above decompositions of $\varphi$ and $\psi$ horizontally compatible. Write $g_1$ as $g_1g^{-1}g=g_1g^{-1}(g_1'\dots g'_s)$. Using this write $i_{g_1}$ as $ i_{g_1g^{-1}}(i_{g'_s}\bullet_\infty\dots \bullet_\infty i_{g'_1})$. Inserting an identity endomorphism in between each $i_{g'_i}$ and $i_{g'_{i+1}}$ in the above decomposition we obtain a decomposition of $i_{g_1}\bullet_\infty\varphi_0$ as:

\[i_{g_1}\bullet_\infty\varphi_0= i_{g_1g^{-1}}\bullet_\infty (i_{g'_s}\bullet_\infty id_{i_{g'_s}}\bullet_\infty \dots \bullet_\infty id_{i_{g'_1}}\bullet_\infty i_{g'_1}\bullet id_{i_{g'_s}})\bullet_\infty \varphi_0\]

\noindent Write $\eta$ for the vertical composition

\[\varphi_{k+1}\bullet_\infty i_{g_k}\dots \bullet_\infty i_{g_2}\bullet_\infty\varphi_1\]

\noindent obtained from decomposition 1 by removing the first two terms from right to left. Substituting in decomposition 1 we obtain a decomposition of $\varphi$ as a vertical composition of the form:

\[(\eta\bullet_\infty i_{g_1g^{-1}})\bullet_\infty(i_{g'_s}\bullet_\infty id_{i_{g'_s}}\bullet_\infty \dots \bullet_\infty id_{i_{g'_1}}\bullet_\infty i_{g'_1}\bullet_\infty \varphi_0) \]

\noindent If we write $\nu_1,\nu_2$ for the expression on the first and second parenthesis above respectively we obtain a pictorial representation of $\varphi$ as:

\begin{center}

\begin{tikzpicture}
\matrix(m)[matrix of math nodes, row sep=4em, column sep=4em,text height=1.5ex, text depth=0.25ex]
{\ast&\ast\\
\ast&\ast\\
\ast&\ast\\};
\path[->,font=\scriptsize]
(m-1-1) edge node [above]{$a$} (m-1-2)
        edge node [left]{$g$} (m-2-1)
(m-2-1) edge[red] node{} (m-2-2)			
(m-1-2) edge node [right]{$g$} (m-2-2)
(m-1-1) edge [white] node[black][fill=white]{$\nu_1$} (m-2-2)

(m-2-1) edge [blue] node {} (m-3-1)
(m-3-1) edge node[below]{$b$} (m-3-2)			
(m-2-2) edge[blue] node {} (m-3-2)
(m-2-1) edge [white] node[black][fill=white]{$\nu_2$} (m-3-2);
\end{tikzpicture}
\end{center}

\noindent  Writing $\psi$ as $i_{id_{b'}}\bullet_\infty \psi$ we obtain a pictorial representation of $\psi$ as:

\begin{center}
\begin{tikzpicture}
\matrix(m)[matrix of math nodes, row sep=4em, column sep=4em,text height=1.5ex, text depth=0.25ex]
{\ast&\ast\\
\ast&\ast\\
\ast&\ast\\};
\path[->,font=\scriptsize]
(m-1-1) edge node [above]{$a'$} (m-1-2)
        edge node [left]{$g$} (m-2-1)
(m-2-1) edge node[below]{$b'$} (m-2-2)			
(m-1-2) edge node [right]{$g$} (m-2-2)
(m-1-1) edge [white] node[black][fill=white]{$\nu_1$} (m-2-2)

(m-2-1) edge [blue] node {} (m-3-1)
(m-3-1) edge node[below]{$b'$} (m-3-2)			
(m-2-2) edge[blue] node {} (m-3-2)
(m-2-1) edge [white] node[black][fill=white]{$id_{b'}$} (m-3-2);
\end{tikzpicture}
\end{center}

\noindent The horizontal composition $\varphi\ast_\infty\psi$ thus admits a pictorial representation as:

\begin{center}

\begin{tikzpicture}
\matrix(m)[matrix of math nodes, row sep=4em, column sep=4em,text height=1.5ex, text depth=0.25ex]
{\ast&\ast&\ast\\
\ast&\ast&\ast\\
\ast&\ast&\ast\\};
\path[->,font=\scriptsize]
(m-1-1) edge node [above]{$a$} (m-1-2)
        edge node [left]{$g$} (m-2-1)
(m-2-1) edge[red] node{} (m-2-2)			
(m-1-2) edge node [right]{$g$} (m-2-2)
(m-1-1) edge [white] node[black][fill=white]{$\nu_1$} (m-2-2)

(m-2-1) edge [blue] node {} (m-3-1)
(m-3-1) edge node[below]{$b$} (m-3-2)			
(m-2-2) edge[blue] node {} (m-3-2)
(m-2-1) edge [white] node[black][fill=white]{$\nu_2$} (m-3-2)

(m-1-2) edge node [above]{$a'$} (m-1-3)
(m-2-2) edge node[below]{$b'$} (m-2-3)			
(m-1-3) edge node [right]{$g$} (m-2-3)
(m-1-2) edge [white] node[black][fill=white]{$\nu_1$} (m-2-3)

(m-2-3) edge[blue] node {} (m-3-3)
(m-3-2) edge node [below]{$b'$}(m-3-3)
(m-2-2) edge [white] node[black][fill=white]{$id_{b'}$} (m-3-3);
\end{tikzpicture}
\end{center}

\noindent By the way $\nu_1$ was defined the horizontal composition of the two upper squares in the above diagram is equal to

\[\psi=\psi_{s+1}\bullet_\infty i_{g'_s}\bullet_\infty\dots \bullet_\infty i_{g'_1}\bullet_\infty(\varphi_0\ast_\infty\psi_0)\]

\noindent which is clear of free vertical length 1. The horizontal composition of the two bottom squares of the above diagram is clearly of free vertical length 1 and thus $\varphi\ast_\infty\psi$ is of vertical length 1. We conclude that $\ell_{free}Q_\mathcal{B}=1$ as desired.
\end{proof}

\noindent We use proposition \ref{groupdecorations} to relate the free globularily generated double category construction to the free product operation between groups. Moreover, we provide an explicit description of saturations of single object 2-categories decorated by deloopings of groups. This is the content of the following corollary.

\begin{corollary}

Let $G,A$ be groups. Suppose $A$ is abelian. Let $\mathcal{B}$ denote the decorated category $(\Omega G,2\Omega A)$. In that case $Q_\mathcal{B}$ has $\Omega G$ as category of objects and $\Omega(G\ast A)$ as category of squares. Moreover, the saturation $H^*Q_\mathcal{B}$ of $\mathcal{B}$ has the subroup of $G\ast A$ of words $a_kg_k\dots a_ig_1$ such that $g_k\dots g_1=1$ as groupoid of globular squares. 
\end{corollary}

\noindent The following example shows that the assumption of $G$ being a group is essential for proposition \ref{groupdecorations}. We show the existence of a single object bicategory $\mathcal{B}$ decorated by a reduced monoid such that $Q_\mathcal{B}$ has squares of length equal to 2.

\begin{ex}\label{inftyN}
Let $\mathcal{B}$ be the decorated bicategory $(\Omega(\mathbb{N}\setminus\left\{1\right\}),2\Omega\mathbb{Z}_2)$. From proposition \ref{reducedprop} and from the fact that $\mathbb{N}\setminus\left\{1\right\}$ is a reduced monoid it follows that $\mathcal{B}$ is saturated. We claim that $ Q_\mathcal{B}$ admits squares of length equal to 2. To see this let $\varphi$ be the following quare:

\begin{center}

\begin{tikzpicture}
\matrix(m)[matrix of math nodes, row sep=4em, column sep=4em,text height=1.5ex, text depth=0.25ex]
{\ast&\ast\\
\ast&\ast\\
\ast&\ast\\
\ast&\ast\\};
\path[->,font=\scriptsize]
(m-1-1) edge [red] node {} (m-1-2)
        edge node [left]{$3$} (m-2-1)
(m-1-2) edge node[right]{$3$} (m-2-2)			
(m-2-1) edge [red] node {} (m-2-2)
(m-1-1) edge [white] node [black][fill=white]{$i_3$}(m-2-2)

(m-2-1) edge [blue] node{} (m-3-1)
(m-2-2) edge [blue]node{} (m-3-2)			
(m-3-1) edge [red]node {} (m-3-2)
(m-2-1) edge [white]node[black][fill=white]{$-1$}(m-3-2)

(m-3-1) edge node[left]{$3$} (m-4-1)
(m-3-2) edge node[right]{$3$} (m-4-2)			
(m-4-1) edge [red]node {} (m-4-2)
(m-3-1) edge [white]node[black][fill=white]{$i_3$}(m-4-2);
\end{tikzpicture}

\end{center}

\noindent and let $\psi$ be the square:

\begin{center}

\begin{tikzpicture}
\matrix(m)[matrix of math nodes, row sep=3em, column sep=3em,text height=1.5ex, text depth=0.25ex]
{\ast&\ast\\
\ast&\ast\\
\ast&\ast\\
\ast&\ast\\
\ast&\ast\\
\ast&\ast\\};
\path[->,font=\scriptsize]
(m-1-1) edge [red] node {} (m-1-2)
        edge node [left]{$2$} (m-2-1)
(m-1-2) edge node[right]{$2$} (m-2-2)			
(m-2-1) edge [red] node {} (m-2-2)
(m-1-1) edge [white] node [black][fill=white]{$i_2$}(m-2-2)

(m-2-1) edge [blue] node{} (m-3-1)
(m-2-2) edge [blue]node{} (m-3-2)			
(m-3-1) edge [red]node {} (m-3-2)
(m-2-1) edge [white]node[black][fill=white]{$-1$}(m-3-2)

(m-3-1) edge node[left]{$2$} (m-4-1)
(m-3-2) edge node[right]{$2$} (m-4-2)			
(m-4-1) edge [red]node {} (m-4-2)
(m-3-1) edge [white]node[black][fill=white]{$i_2$}(m-4-2)

(m-4-1) edge[blue] node{} (m-5-1)
(m-4-2) edge[blue] node[right]{} (m-5-2)			
(m-5-1) edge [red]node {} (m-5-2)
(m-4-1) edge [white]node[black][fill=white]{$-1$}(m-5-2)

(m-5-1) edge node[left]{$2$} (m-6-1)
(m-5-2) edge node[right]{$2$} (m-6-2)			
(m-6-1) edge [red]node {} (m-6-2)
(m-5-1) edge [white]node[black][fill=white]{$i_2$}(m-6-2);
\end{tikzpicture}

\end{center}

\noindent Thus defined both $\varphi$ and $\psi$ are horizontal endomorphisms in $V_1$ such that $t_\infty\varphi=6=s_\infty\psi$. By the fact that both 2 and 3 are irreducible in $\mathbb{N}\setminus\left\{1\right\}$ it easily follows that $\varphi\ast_\infty\psi$ and any horizontal composition of $\varphi\ast_\infty\psi$ with itself are not  morphisms in $V_1$ and are thus of free length $\geq 2$. Clearly both $\varphi\ast_\infty\psi$ and and horizontal composition of $\varphi\ast_\infty\psi$ with itself are of free legth $\leq 2$ ans thus are of free length exactly 2. 
\end{ex}

\noindent Observe that the arguments of subdividing squares of free length 1 employed in the proof of proposition \ref{groupdecorations} can easily be modified to to prove that the free globularily generated double category associated to any monoidal category decorated by $\Omega\mathbb{N}$ has vertical length 1. Moreover, observe that in the case in which a decorated bicategory $\mathcal{B}$ is of the form $(\Omega M,A)$ for monoids/algebras $M,A$ where $A$ is commutative, then the first term of the free vertical filtration $V_1$ of $Q_\mathcal{B}$ is equal to the delooping $\Omega(M\ast A)$ of $M\ast A$.

\section{von Neumann algebras}\label{s6}

\noindent In this section we study applications of the free globularily generated double category construction to the problem of existence of functorial extensions of the Haagerup standard form construction and the Connes fusion operation, see \cite{Bartels1}. We prove that the bicategory of factors, Hilbert bimodules, and intertwining operators, decorated by not-necessarily finite index inclusions is saturated. This provides extensions of the Haagerup standard form construction and the Connes fusion operation, on the category of factors and not-necessarily finite index inclusions and a certain linear category properly containing the category of Hilbert spaces and bounded operators. These functors are compatible in the sense that they form the structure data of a category internal to linear categories internalizing the decorated bicategory of factors. We apply the saturation process introduced in section \ref{s3} to the problem of extending the Haagerup standard form construction and the Connes fusion operation, to functors on a category of general (not-necessarily factors) von Neumann algebras and general (not-necessarily finite index) von Neumann algebra morphisms. 

Our construction is as follows: We write \textbf{Mod}$^{fact}$ for the bicategory whose 2-cells are of the form:

\begin{center}

\begin{tikzpicture}
\matrix(m)[matrix of math nodes, row sep=4em, column sep=4em,text height=1.5ex, text depth=0.25ex]
{A&B\\};
\path[->,font=\scriptsize]
(m-1-1) edge [bend left=45] node [above]{$H$}(m-1-2)
        edge [bend right=45] node[below]{$K$}(m-1-2)
        edge [white]node[black][fill=white]{$\varphi$}(m-1-2);
\end{tikzpicture}
\end{center}

\noindent where $A,B$ are factors, $H,K$ are $A$-$B$ left-right Hilbert bimodules over $A,B$ and where $\varphi$ is a bounded intertwiner from $H$ to $K$. The horizontal identity cells in \textbf{Mod}$^{fact}$ are of the form:

\begin{center}

\begin{tikzpicture}
\matrix(m)[matrix of math nodes, row sep=4em, column sep=4em,text height=1.5ex, text depth=0.25ex]
{A&A\\};
\path[->,font=\scriptsize]
(m-1-1) edge [red,bend left=45] node[black] [above]{$L^2(A)$}(m-1-2)
        edge [red,bend right=45] node[black][below]{$L^2(A)$}(m-1-2)
        edge [white]node[black][fill=white]{$id_{L^2(A)}$}(m-1-2);
\end{tikzpicture}
\end{center}

\noindent where $A$ is a factor and where $L^2(A)$ denotes the Haagerup standard form of $A$, see \cite{Haagerup,Bartels1}. Given two horizontally compatible 2-cells in \textbf{Mod}$^{fact}$ of the form:

\begin{center}

\begin{tikzpicture}
\matrix(m)[matrix of math nodes, row sep=4em, column sep=4em,text height=1.5ex, text depth=0.25ex]
{A&B&C\\};
\path[->,font=\scriptsize]
(m-1-1) edge [bend left=45] node [above]{$H$}(m-1-2)
        edge [bend right=45] node[below]{$K$}(m-1-2)
        edge [white]node[black][fill=white]{$\varphi$}(m-1-2)
        
  (m-1-2) edge [bend left=45] node [above]{$H'$}(m-1-3)
        edge [bend right=45] node[below]{$K'$}(m-1-3)
        edge [white]node[black][fill=white]{$\varphi'$}(m-1-3)      ;
\end{tikzpicture}
\end{center}

\noindent the horizontal composition of $\varphi$ and $\varphi'$ in \textbf{Mod}$^{fact}$ is the 2-cell:

\begin{center}

\begin{tikzpicture}
\matrix(m)[matrix of math nodes, row sep=5em, column sep=4em,text height=1.5ex, text depth=0.25ex]
{A&C\\};
\path[->,font=\scriptsize]
(m-1-1) edge [bend left=45] node[black] [above]{$H\boxtimes_B H'$}(m-1-2)
        edge [bend right=45] node[black][below]{$K\boxtimes_B K'$}(m-1-2)
        edge [white]node[black][fill=white]{$\varphi\boxtimes_B\varphi'$}(m-1-2);
\end{tikzpicture}
\end{center}

\noindent where $H\boxtimes_BH',K\boxtimes_BK'$ and $\varphi\boxtimes_B\varphi'$ denote the Connes fusion of $H$ and $H'$, of $K$ and $K'$ and of $\varphi$ and $\varphi'$ respectively. Thus defined \textbf{Mod}$^{fact}$ is linear ($C^*$ tensor in fact).  We write \textbf{vN}$^{fact}$ for the category whose objects are factors and whose morphisms are (possibly infinite index) von Neumann algebra morphisms. Thus defined \textbf{vN}$^{fact}$ is linear. The pair $(\mbox{\textbf{vN}}^{fact}, \mbox{\textbf{Mod}}^{fact})$ is thus a linear decorated bicategory. We write $W^*_{fact}$ for this decorated bicategory. We prove the following proposition.

\begin{prop}\label{HaagerupExtension}
The linear decorated bicategory $W^*_{fact}$ is saturated and moreover, the equation $\ell Q^{\mathbb{C}}_{W^*_{fact}}=1$ holds.
\end{prop}

\begin{proof}
We wish to prove that $W^*_{fact}$ is saturated and that it satisfies the equation $\ell Q^{\mathbb{C}}_{W^*_{fact}}=1$. We prove that every square in $V_1$ is a multiple of a square admitting a pictorial representation as:

\begin{center}

\begin{tikzpicture}
\matrix(m)[matrix of math nodes, row sep=4em, column sep=4em,text height=1.5ex, text depth=0.25ex]
{\bullet&\bullet\\
\bullet&\bullet\\
\bullet&\bullet\\
\bullet&\bullet\\};
\path[->,font=\scriptsize]
(m-1-1) edge node {} (m-1-2)
        edge [blue]node {} (m-2-1)
(m-1-2) edge [blue]node {} (m-2-2)			
(m-2-1) edge [red] node {} (m-2-2)
(m-1-1) edge [white] node [black][fill=white]{}(m-2-2)

(m-2-1) edge node {} (m-3-1)
(m-2-2) edge node{} (m-3-2)			
(m-3-1) edge [red]node {} (m-3-2)
(m-2-1) edge [white]node[black][fill=white]{}(m-3-2)

(m-3-1) edge [blue]node{} (m-4-1)
(m-3-2) edge [blue]node{} (m-4-2)			
(m-4-1) edge node {} (m-4-2)
(m-3-1) edge [white]node[black][fill=white]{}(m-4-2);
\end{tikzpicture}

\end{center}

\noindent Let $\varphi$ be a square in $Q^\mathbb{C}_{W^*_{fact}}$ of free length 1. Represent $\varphi$ pictorially as:

\begin{center}

\begin{tikzpicture}
\matrix(m)[matrix of math nodes, row sep=4em, column sep=4em,text height=1.5ex, text depth=0.25ex]
{A&A\\
B&B\\};
\path[->,font=\scriptsize]
(m-1-1) edge node [above]{$\alpha$} (m-1-2)
        edge node [left]{$f$} (m-2-1)
(m-2-1) edge node[below]{$\beta$} (m-2-2)			
(m-1-2) edge node[right]{$f$} (m-2-2)
(m-1-1) edge [white]node[black][fill=white]{$\varphi$}(m-2-2);
\end{tikzpicture}
\end{center}

\noindent Write $\varphi$ as a vertical composition of the form

\[\psi_{k+1}\bullet_\infty i_{f_k}\bullet_\infty\psi_{k-1}\bullet_\infty\dots\bullet_\infty\psi_1\bullet_{\infty}i_{f_1}\bullet_\infty\psi_0\]

\noindent where $f_i$ is morphism from a factor $A_{i-1}$ to a factor $A_i$, where $f$ admits a decomposition as $f=f_k\dots f_1$, and where $\psi_i$ is a square of the form:

\begin{center}

\begin{tikzpicture}
\matrix(m)[matrix of math nodes, row sep=4em, column sep=4em,text height=1.5ex, text depth=0.25ex]
{A_i&A_i\\
A_i&A_i\\};
\path[->,font=\scriptsize]
(m-1-1) edge [red]node {} (m-1-2)
        edge [blue]node{} (m-2-1)
(m-2-1) edge [red]node{} (m-2-2)			
(m-1-2) edge [blue]node{} (m-2-2)
(m-1-1) edge [white] node [black][fill=white] {$\psi_i$}(m-2-2);
\end{tikzpicture}
\end{center}

\noindent for every $1\leq i\leq k$ and where $\psi_0,\psi_{k+1}$ are squares of the form:

\begin{center}

\begin{tikzpicture}
\matrix(m)[matrix of math nodes, row sep=4em, column sep=4em,text height=1.5ex, text depth=0.25ex]
{A&A&B&B\\
A&A&B&B\\};
\path[->,font=\scriptsize]
(m-1-1) edge node [above]{$\alpha$} (m-1-2)
        edge [blue]node {} (m-2-1)
(m-2-1) edge [red]node{} (m-2-2)			
(m-1-2) edge [blue]node{} (m-2-2)
(m-1-1) edge [white] node[black][fill=white]{$\psi_0$} (m-2-2)

(m-1-3) edge [red]node {} (m-1-4)
        edge [blue]node {} (m-2-3)
(m-2-3) edge node[below]{$\beta$} (m-2-4)			
(m-1-4) edge [blue]node {} (m-2-4)
(m-1-3) edge [white] node[black][fill=white]{$\psi_1$} (m-2-4);
\end{tikzpicture}
\end{center}

\noindent Let $1\leq i\leq k$. From the fact that $A_i$ is a factor it follows that the algebra of endomorphisms $End_{W^*_{fact}}(L^2(A_i))$ of $A_i$ is 1-dimensional and thus is equal to $\mathbb{C}id_{L^2(A_i)}$. From this it follows there exists a $\lambda_i\in\mathbb{C}$ such that the square $\psi_i$ is equal to $\lambda_i$ times the square:

\begin{center}

\begin{tikzpicture}
\matrix(m)[matrix of math nodes, row sep=4em, column sep=4em,text height=1.5ex, text depth=0.25ex]
{A_i&A_i\\
A_i&A_i\\};
\path[->,font=\scriptsize]
(m-1-1) edge [red]node {} (m-1-2)
        edge [blue]node{} (m-2-1)
(m-2-1) edge [red]node{} (m-2-2)			
(m-1-2) edge [blue]node{} (m-2-2)
(m-1-1) edge [white] node [black][fill=white] {$id_{L^2(A_i)}$}(m-2-2);
\end{tikzpicture}
\end{center}

\noindent From this we conclude that $\varphi$ is equal to $\Pi_{i=1}^k\lambda_i$ times the square:

\begin{center}

\begin{tikzpicture}
\matrix(m)[matrix of math nodes, row sep=4em, column sep=4em,text height=1.5ex, text depth=0.25ex]
{A&A\\
A&A\\
B&B\\
B&B\\};
\path[->,font=\scriptsize]
(m-1-1) edge node [above]{$\alpha$} (m-1-2)
        edge [blue]node {} (m-2-1)
(m-1-2) edge [blue]node {} (m-2-2)			
(m-2-1) edge [red] node {} (m-2-2)
(m-1-1) edge [white] node [black][fill=white]{$\psi_0$}(m-2-2)

(m-2-1) edge node [left]{$f$} (m-3-1)
(m-2-2) edge node[right]{$f$} (m-3-2)			
(m-3-1) edge [red]node {} (m-3-2)
(m-2-1) edge [white]node[black][fill=white]{$i_f$}(m-3-2)

(m-3-1) edge [blue]node{} (m-4-1)
(m-3-2) edge [blue]node{} (m-4-2)			
(m-4-1) edge node [below]{$\beta$} (m-4-2)
(m-3-1) edge [white]node[black][fill=white]{$\psi_1$}(m-4-2);
\end{tikzpicture}

\end{center}

\noindent This proves our claim. Observe that a square of the form above is globular if and only if it is a square in $W^*_{free}$. This proves that $H^*Q^\mathbb{C}_{W^*_{free}}=W^*_{free}$. The equation $\ell Q^\mathbb{C}_{W^*_{free}}=1$ follows from the fact that the horizontal composition of two squares admitting pictorial representations as above admits a pictorial representation as above. This concludes the proof of the proposition.
\end{proof}

\noindent The category of squares $Q^\mathbb{C}_{W^*_{free_1}}$ of $Q^\mathbb{C}_{W^*_{free}}$ is thus a linear category whose objects are Hilbert bimodules between factors, whose morphisms are either usual intertwining operators between Hilbert bimodules or formal compositions as described in the proof of proposition \ref{HaagerupExtension}. The function associating to every von Neumann algebra $A$ its Haagerup standard form $L^2(A)$ admits an extension, as the horizontal identity functor of $Q_{W^*_{free}}$, to a linear functor

\[L^2:\mbox{\textbf{vN}}^{fact}\to Q^\mathbb{C}_{W^*_{free_1}}\]

\noindent The functor (on the left and right entries) associating to every compatible pair of Hilbert bimodules $H,K$ or intertwining operators $\varphi,\psi$ their Connes fusion $H\boxtimes K$ or $\varphi\boxtimes\psi$ respectively, admits an extension (to the bottom variable), as the horizontal composition functor of $Q_{W^*_{fact}}$, to a linear functor

\[\boxtimes_\bullet:Q^\mathbb{C}_{W^*_{free_1}}\times_{\mbox{\textbf{vN}}^{fact}}Q^\mathbb{C}_{W^*_{free_1}}\to Q^\mathbb{C}_{W^*_{free_1}}\]

\noindent Moreover, these two linear functors are compatible in the sense that they provide $Q^\mathbb{C}_{W^*_{free}}$ with the structure of a category internal to linear categories.

The techniques employed in the proof of proposition \ref{HaagerupExtension} do not apply to the bicategory of general, i.e. non-cessarily factor, von Neumann algebras nor even to semisimple von Neumann algebras. Through corollary \ref{saturationcorollary} we obtain weaker versions of proposition \ref{HaagerupExtension} for the case of von Neumann algebras with not-necessarily trivial center. Write \textbf{Mod} for the linear bicategory whose 2-cells are of the form:

\begin{center}

\begin{tikzpicture}
\matrix(m)[matrix of math nodes, row sep=4em, column sep=4em,text height=1.5ex, text depth=0.25ex]
{A&B\\};
\path[->,font=\scriptsize]
(m-1-1) edge [bend left=45] node [above]{$H$}(m-1-2)
        edge [bend right=45] node[below]{$K$}(m-1-2)
        edge [white]node[black][fill=white]{$\varphi$}(m-1-2);
\end{tikzpicture}
\end{center}

\noindent where $A,B$ are now general von Neumann algebras, $H,K$ are left-right Hilbert bimodules over $A,B$, and $\varphi$ is an intertwining operator from $H$ to $K$. The horizontal identity and horizontal composition on \textbf{Mod} are defined in analogy to those defining \textbf{Mod}$^{fact}$. Write \textbf{vN} for the the linear category of von Neumann algebras and general (not-necessarily finite) von Neumann algebra morphisms. The pair $(\mbox{\textbf{vN}},\mbox{\textbf{Mod}})$ is a decorated linear category. We write $W^*$ for this decorated bicategory. Write $\tilde{W}^*$ for the saturation of $W^*$. In that case the category of morphisms $Q^\mathbb{C}_{\tilde{W}^*_1}$ of $Q^\mathbb{C}_{\tilde{W}^*}$ is a linear category, whose objects are Hilbert bimodules over general von Neumann algebras, and whose morphisms contain the usual intertwining operators in \textbf{Mod}. The function associating to every von Neumann algebra $A$ its Haagerup standard form $L^2(A)$ extends, as the horizontal identity functor of $Q^\mathbb{C}_{\tilde{W}^*}$, to a linear functor

\[L^2:\mbox{\textbf{vN}}\to Q^\mathbb{C}_{\tilde{W}^*_1}\]

\noindent and the Connes fusion bifunctor $\boxtimes$ extends, as the horizontal composition functor of $Q^\mathbb{C}_{\tilde{W}^*}$, to a linear functor

\[\boxtimes_\bullet:Q^\mathbb{C}_{\tilde{W}^*_1}\times_{\mbox{\textbf{vN}}}Q^\mathbb{C}_{\tilde{W}^*_1}\to Q^\mathbb{C}_{\tilde{W}^*_1}\]

\noindent Moreover, these functors are compatible in the sense that they provide $Q^\mathbb{C}_{\tilde{W}^*}$ with the structure of a linear double category.

In \cite{Bartels1,Bartels2} a solution to problem \ref{prob} is presented for the decorated bicategory whose 2-cells are of the form:

\begin{center}

\begin{tikzpicture}
\matrix(m)[matrix of math nodes, row sep=4em, column sep=4em,text height=1.5ex, text depth=0.25ex]
{A&B\\};
\path[->,font=\scriptsize]
(m-1-1) edge [bend left=45] node [above]{$H$}(m-1-2)
        edge [bend right=45] node[below]{$K$}(m-1-2)
        edge [white]node[black][fill=white]{$\varphi$}(m-1-2);
\end{tikzpicture}
\end{center}

\noindent where $A,B$ are factors (more generally $A,B$ are semisimple) $H,K$ are left-right Hilbert bimodules over $A,B$ and where $\varphi$ is an intertwiner operator from $H$ to $K$, and whose decoration is the category of factors and finite index inclusions. The horizontal identity and the horizontal composition functors, i.e. the corresponding functorial extensions of the Haagerup standard form construction and the Connes fusion operation, are defined making strong use of the Kosaki theory of minimal conditional expectations of finite index subfactors \cite{Kosaki}. We write $BDH$ for this double category. We ask how the double category $Q^\mathbb{C}_{W^*_{fact}}$ described in the proof of proposition \ref{HaagerupExtension} and $BDH$ are related. We consider the sub-double category of $Q^\mathbb{C}_{W^*_{fact}}$ generated by globular squares and the squares of the form:

\begin{center}

\begin{tikzpicture}
\matrix(m)[matrix of math nodes, row sep=4em, column sep=4em,text height=1.5ex, text depth=0.25ex]
{A&A\\
B&B\\};
\path[->,font=\scriptsize]
(m-1-1) edge node [above]{$H$} (m-1-2)
        edge node [left]{$f$} (m-2-1)
(m-2-1) edge node[below]{$K$} (m-2-2)			
(m-1-2) edge node [right]{$f$} (m-2-2)
(m-1-1) edge [white]node[black][fill=white]{$\varphi$}(m-2-2) ;
\end{tikzpicture}
\end{center}

\noindent where $f$ is an inclusion of finite Jones index. We write $Q^\mathbb{C}_{W^*_{fin}}$ for this double category and we write $W^*_{fin}$ for $H^*Q^\mathbb{C}_{W^*_{fin}}$. We have the following equation:

\[H^*Q^\mathbb{c}_{W^*_{fin}}=W^*_{fin}=H^*BDH\]

\noindent and thus $Q^\mathbb{C}_{W^*_{fin}}$ and $BDH$ have the same category of objects, the same collection of horizontal morphisms, and the same collection of horizontal and globular squares. It is natural to expect some higher relation between the squares of $Q^\mathbb{C}_{W^*_{fin}}$ and the squares of $\gamma BDH$ to hold. It is easily seen that certain relations that hold in $\gamma BDH$ do not hold on $Q^\mathbb{C}_{W^*_{fin}}$, e.g. change of base algebra. This makes it obvious that the double categories $\gamma BDH$ and $Q^\mathbb{C}_{W^*_{fin}}$ are non-equivalent. There is an obvious strict tensor double functor $\pi$ from $Q^\mathbb{C}_{W^*_{fin}}$ to $\gamma BDH$ such that $\pi$ restricts to the identity on $H^*Q^\mathbb{C}_{W^*_{fin}}$. This double functor preserves squares of the form:

\begin{center}

\begin{tikzpicture}
\matrix(m)[matrix of math nodes, row sep=4em, column sep=4em,text height=1.5ex, text depth=0.25ex]
{A&A\\
B&B\\};
\path[->,font=\scriptsize]
(m-1-1) edge [red]node {} (m-1-2)
        edge node [left]{$f$} (m-2-1)
(m-2-1) edge [red]node{} (m-2-2)			
(m-1-2) edge node [right]{$f$} (m-2-2)
(m-1-1) edge [white]node[black][fill=white]{$L^2(f)$}(m-2-2) ;
\end{tikzpicture}
\end{center}

\noindent Since both $Q^\mathbb{C}_{W^*_{fin}}$ to $\gamma BDH$ are generated by both $H^*Q^\mathbb{C}_{W^*_{fin}}$ and the set of squares as above, the double functor $\pi$ is unique with respect to its value on $H^*Q^\mathbb{C}_{W^*_{fin}}$ and is surjective on squares. We study double functors of this form and the way they relate free globularily generated double categories to globularily generated internalizations in the second installment of the present series of papers.

The constructions presented above have an obvious drawback. All the categories of von Neumann algebras and all the bicategories of Hilbert bimodules we have considered are symmetric monoidal. We wish for the corresponding free globularily generated double categories and thus for the corresponding functorial extensions of the Haagerup standard form construction and the Connes fusion operation to be symmetric monoidal. It is not obvious how to extend the combined coherence data of both the decoration and the undelying bicategory of a symmetric monoidal bicateogry into coherence data for the obvious choice of monoidal structure on the free globularily generated double category construction. These questions will be explored elsewhere.

\bibliographystyle{plain}
\bibliography{freeglobular}

\end{document}